\documentclass[11pt,twoside]{amsart}
\usepackage{mathrsfs}
\usepackage{txfonts}
\usepackage{bbm}
\usepackage{amsfonts}
\usepackage{amsmath}
\usepackage{amssymb}
\usepackage{wasysym}
\usepackage{psfrag}
\usepackage{graphics,epsfig,amsmath}  
\usepackage{comment}
\usepackage{enumerate}
\usepackage{color}
\usepackage{fancyhdr}
\usepackage{diagbox}
\usepackage{threeparttable}
\usepackage{tabularx}
\usepackage{booktabs}

\renewcommand\theequation{\thesection.\arabic{equation}}
\newtheorem{Theorem}{Theorem}[section]
\newtheorem{Lemma}[Theorem]{Lemma}
\newtheorem{Proposition}{Proposition}[section]

\newtheorem{remark}[Theorem]{Remark}

\usepackage{metalogo}
\usepackage{cite}
\usepackage[pagewise]{lineno}

\allowdisplaybreaks

\selectfont

\renewcommand{\emptyset}{\varnothing}

\title[Multiple solutions for higher order fractional Laplace equation]{Multiple solutions \\ for elliptic equations driven by higher order fractional Laplacian}

\author[F. Cheng]{Fuwei Cheng}
\address{School of Mathematical Sciences, Beijing Normal University, No. 19, XinJieKouWai St., HaiDian District, Beijing 100875, P. R. China}
\email{fwcheng@mail.bnu.edu.cn}

\author[X. Su]{Xifeng Su}
\address{School of Mathematical Sciences, Laboratory of Mathematics and Complex Systems (Ministry of Education)\\
	Beijing Normal University,
	No. 19, XinJieKouWai St., HaiDian District, Beijing 100875, P. R. China}
\email{xfsu@bnu.edu.cn, billy3492@gmail.com}

\author[J. Zhang]{Jiwen Zhang}
\address{School of Mathematical Sciences, Beijing Normal University, No. 19, XinJieKouWai St., HaiDian District, Beijing 100875, P. R. China}
\address{Department of Mathematics and Statistics, University of Western Australia, 35 Stirling Highway, WA 6009 Crawley, Australia}
\email{jwzhang628@mail.bnu.edu.cn, jiwen.zhang@uwa.edu.au}

\begin{document}
	
	\begin{abstract}
		We consider an elliptic partial differential equation driven by higher order fractional Laplacian $(-\Delta)^{s}$, $s \in (1,2)$ with homogeneous Dirichlet boundary condition
		\begin{equation*}
			\left\{%
			\begin{array}{ll}
				(-\Delta)^{s} u=f(x,u) & \text{ in }\Omega, \\
				u=0 &  \text{ in } \mathbb{R}^n \setminus \Omega. \\
			\end{array}%
			\right.
		\end{equation*}
		
	  The above equation has a variational nature, and we investigate the existence and multiplicity results for its weak solutions under various conditions on the nonlinear term $f$:  superlinear growth, concave-convex and symmetric conditions and their combinations.
	    
	  The existence  of two different non-trivial weak solutions is established by Mountain Pass Theorem and Ekeland's variational principle, respectively. 
	  Furthermore, due to Fountain Theorem and its dual form, both infinitely many weak solutions with positive energy and infinitely many weak solutions with negative energy are obtained.
	\end{abstract}

	\maketitle
	\noindent{\it Keywords:} higher order fractional Laplacian, Mountain Pass Theorem, Ekeland's variational method, Fountain Theorem, Dual Fountain Theorem.
	
	\thanks{The datasets generated and/or analysed during the current study are available from the corresponding author on reasonable request.}
	
	%\tableofcontents

	\section{Introduction}
	It is well-known that any positive power $s$ of the  Laplacian can be thought of as a pseudo-differential operator which can be defined via the Fourier transform $\mathscr{F}$ in the following way
	 (see \cite{MR1347689} for instance), 
	 	\begin{equation*}
		(-\Delta)^{s} \varphi(x)=\mathscr{F}^{-1}\left(|\cdot|^{2 s} \mathscr{F}(\varphi)\right)(x) \quad \text { for all } \varphi \in C_{c}^{\infty}\left(\mathbb{R}^{n}\right).
	\end{equation*}\par 
	
		We will consider mainly the existence and multiplicity results of weak solutions to  the following elliptic problem driven by the nonlocal  fractional Laplacian with homogeneous Dirichlet boundary condition: 
	\begin{equation}\label{problem}
		\left\{%
		\begin{array}{ll}
			(-\Delta)^{s} u=f(x,u)   &\text{in } \Omega, \\
			 u=0   &\text{in }\mathbb{R}^n \setminus \Omega 
		\end{array}%
		\right.
	\end{equation} 
	where $0<s\notin \mathbb{Z}$, $n>2s$ and $\Omega \subset \mathbb{R}^{n}$ is an open bounded set with sufficiently smooth boundary.
	
	For $s\in (0,1)$, there is a great attention dedicated to the so-called standard fractional Laplacian (to describe the $2s$-stable L\'evy process), such as the thin obstacle problem \cite{Silvestre07,CSS08}, minimal surface \cite{CR10,CV11}, phase transitions \cite{SV09} and so on. See \cite{DPV12} and references therein  for an introduction to the literature.	
	The results for the existence and multiplicity of weak solutions are well established in the framework of variational analysis, see for example \cite{SV12, WS15}. See also \cite{SVWZ24} for  more general operators.
		
	For $1<s \notin \mathbb{Z}$, the operators $(-\Delta)^{s}$ can be viewed as the nonlocal counterparts of polyharmonic operators. The Higher-order fractional Laplacians arise for example in geometry in connection to the fractional Paneitz operator on the hyperbolic space \cite{MR1965361,MR2737789}, in the theory of the Navier-Stokes equation  as a hyper-dissipative term  \cite{MR1911664, MR2603802}, and in generalizations of the Lane-Emden equations \cite{MR3636636}. 
	
	We would remark that the combination of nonlocality and polyharmonicity really poses new challenges --- for instance, it is not clear so far that the $L^{\infty}$~boundedness is valid in this setting and whether the methods of Nehari manifold may apply. One may refer to \cite{Abatangelo22} and references therein for a rather complete survey of recent results for higher order fractional Laplacians.
	
	Following the spirit of \cite{DG17} for instance,  we will adopt the idea of the composition of operators to obtain the equivalent definition of $(-\Delta)^{s}$ below: % \cite{MR3856149}, \cite{Silvestre07}, 
	for any given $s=m+\sigma$ where $m \in \mathbb{Z}_{+}, \sigma \in (0,1)$, we have
	 \[(-\Delta)^{s} u(x)=(-\Delta)^{m}(-\Delta)^{\sigma} u(x) \]
	for all $u \in C^{2\beta}(\Omega) \cap \mathscr{L}_{\sigma}^{1}$ with $\beta>s $, where 
	\[ \mathscr{L}_{\sigma}^{1}:=\left\{u \in L_{\text {loc }}^{1}(\mathbb{R}^{n}): \int_{\mathbb{R}^{n}} \frac{|u(x)|}{1+|x|^{n+2 \sigma}} d x<\infty\right\}.\]
	Here $(-\Delta)^{m}$ is the $m$-Laplacian
	%the so called polyharmonic operator \cite{GGS91} 
	given by
	\[(-\Delta)^{m} u:=\left(-\sum_{i=1}^{n} \partial_{i i}\right)^{m} u\]
	and $(-\Delta)^\sigma$ is the standard fractional Laplacian given by
	 \[
	  (-\Delta)^{\sigma } u(x) :=C_{n, \sigma }\lim_{\varepsilon\searrow0} \int_{\mathbb{R}^{n}\setminus B_\varepsilon(x)}  \frac{u(x)-u(y)}{|x-y|^{n+2\sigma} }dy, \quad \text{ with } C_{n, \sigma } =-\frac{2^{2 \sigma} \Gamma(n / 2+\sigma)}{\pi^{n / 2}\Gamma(-\sigma )}. 
	  \]
	 	
	Without loss of generality, in this paper, we will fix the exponent $s \in (1,2)$  in problem~\eqref{problem} to avoid unnecessary technicalities. In fact, it is rather immediate to obtain the present results for $2<s\in \mathbb{Z}$ with several modifications on the spaces of functions, energy functionals, etc.
		
	Our main goal is to show the existence and multiplicity results for weak solutions of the Dirichlet problem \eqref{problem} with $s\in (1,2)$  under various conditions on the nonlinear term $f$:  superlinear growth, concave-convex, symmetric conditions and their combinations. This would be a starting point for us to carry on the further studies on the existence, multiplicity and regularity in the future works.
			
	To begin with,  we give the following standard assumptions of nonlinear analysis (see \cite{MIYAGAKI08, WS15} for instance).
			Assume that $f: \overline{\Omega} \times \mathbb{R} \rightarrow \mathbb{R}$ is a continuous function verifying the following conditions:
	\begin{itemize}
		\item [\textbf{(H1)}] there exist $a_{1}, a_{2}>0$ and $q \in\left(2,2^{*}_{s}\right)$, where $2^*_s:=\frac{2n}{n-2s}$, such that
		\[|f(x, t)| \leqslant a_{1}+a_{2}|t|^{q-1} \quad \text { a.e. } x \in \Omega,~~ t \in \mathbb{R} ;\]
		\item[\textbf{(H2)}] $\lim \limits_{|t| \rightarrow 0} \frac{f(x, t)}{|t|}=0$ uniformly in $ x \in \Omega$;
		\item[\textbf{(H3)}] $\lim \limits_{|t| \rightarrow +\infty} \frac{F(x, t)}{t^{2}}=+\infty$ uniformly for a.e. $x \in \Omega$, where $F(x, t):=\int_{0}^{t} f(x, \tau) d \tau$;
		\item[\textbf{(H4)}] there exists $T_{0}>0$ such that for any $x \in \Omega$, the function
		\begin{center}
			 $t \mapsto \frac{f(x, t)}{t}$ is increasing in $t > T_{0}$, and decreasing in $t < -T_{0}$.
		\end{center}
	\end{itemize}    
	
	We remark that (H3), (H4)  are a bit weaker assumptions (several specific examples are given in Theorem~\ref{Thm: application}) than the following Ambrosetti-Rabinowitz condition \cite{AR73}:
	\begin{itemize}
	\item [\textbf{(AR)}]
	There exist $\zeta>2$ and $r>0$ such that a.e. $x \in \Omega$,  $t \in \mathbb{R},~~|t| \geqslant r$		
	\[0<\zeta F(x, t) \leqslant t f(x, t).\]
	\end{itemize}

    Now,  we state our first result. 	
	\begin{Theorem}\label{Mountain Pass sol theorem}
		Assume $s\in (1,2)$. Let $f$ be a continuous function verifying (H1)-(H4).
		Then, the superlinear problem~\eqref{problem} admits a non-trivial Mountain Pass solution $u \in \mathcal{H}_{0}^{s}$ with positive energy $\mathcal{J}(u) $ (see \eqref{J} below). 
	\end{Theorem}
	 
	Inspired by \cite {Ambrosetti1994CombinedEO} for the case of the standard Laplacian and \cite {WS15} for the case of the fractional Laplacian, we consider the following more concrete concave-convex nonlinear problem
	\begin{equation}\label{problem2}
		\left\{
		\begin{array}{ll}
			(-\Delta)^{s} u=\lambda|u|^{p-2}u+\mu g(x,u)=:f_{\lambda, \mu}(x,u) & \text{in }\Omega, \\
			u=0 &  \text{in }\mathbb{R}^n\setminus \Omega \\
		\end{array}
		\right.
	\end{equation}
	where $1<p<2$, $\lambda, \mu \in \mathbb{R}$ are two parameters. 
	Here and in the sequel, we denote by $\mathcal{J}_{\lambda, \mu }$ the energy functional corresponding to problem \eqref{problem2}  and $\mathcal{J}_{\lambda}:=\mathcal{J}_{\lambda, 1}$ for simplicity. 
	 
	We then obtain the following theorem by Ekeland's variational principle.
	\begin{Theorem}\label{negative energy sol theorem}
		Assume that $s\in (1,2),\, \mu=1$ and  $g$ is a continuous function verifying (H1)-(H4). Then there exists $\lambda^{*}>0$ such that for $\lambda \in\left(0, \lambda^{*}\right)$, the concave-convex problem \eqref{problem2} has at least two non-trivial weak solutions $u,v$ satisfying $\mathcal{J_{\lambda}}(u)>0>\mathcal{J_{\lambda}}(v)$.
	\end{Theorem}
	Our next goal is to show the existence of infinitely many solutions for both the superlinear problem \eqref{problem} and the concave-convex problem \eqref{problem2} by imposing the symmetry condition on the nonlinearity. Namely,  we assume further that $f$ satisfies
	\begin{itemize}
	\item [\textbf{(S)}] 
	$f(x,-t)=-f(x,t)$ \quad for any $x \in \Omega,\, t\in \mathbb{R}$.
	\end{itemize}
	
	Owing to the well-known Fountain Theorem and its dual form, which were first established in \cite{Bartsch93} and \cite{BW95} respectively  (see also \cite{Willembook}), we  have the following two corresponding theorems for the superlinear problem \eqref{problem} and the concave-convex problem \eqref{problem2} respectively.
	\begin{Theorem}\label{infinitely many sol Theorem}
		Assume $s\in (1,2)$. Let $f$ be a continuous function verifying (H1), (H3), (H4) and (S).
		Then, the superlinear problem \eqref{problem} admits infinitely many weak solutions $\{u_{j}\}_{j \in \mathbb{N}} \subset \mathcal{H}_{0}^{s}$ such that $\mathcal{J}(u_j) \rightarrow +\infty$ as $j \rightarrow +\infty$.
	\end{Theorem}

	\begin{Theorem}\label{infinite positive and negative energy sol theorem}
		Let $s\in(1,2), \,\lambda,\, \mu >0$ and $g$ be a continuous function verifying (H1)-(H4) and (S). Then, the concave-convex problem \eqref{problem2} has both a sequence of weak solutions $\{u_k\}$ such that $\mathcal{J}_{\lambda,\mu}(u_k)\rightarrow +\infty$ and a sequence of solutions $\{v_k\}$ such that  $\mathcal{J}_{\lambda,\mu}(v_k)\rightarrow 0^-$, as $k\rightarrow +\infty$.
	\end{Theorem}
	
	As a supplement, when 
	$g(x,u)$ takes the special form $g(x,u)=|u|^{q-2}u$
	, we could choose real parameters $\lambda,\mu$ (not necessarily positive), and prove the following theorem, which could not be covered by the above theorems.
	\begin{Theorem}\label{infinite positive and negative energy sol theorem-2}
	Let $g(x,u)=|u|^{q-2}u$. Assume $s\in (1,2), \, 1<p<2<q<2^{*}_{s}$. Then, we have
	\begin{itemize}
			\item[(a)] for every $\mu>0$, $\lambda\in \mathbb{R}$, the concave-convex problem \eqref{problem2} has a sequence of weak solutions $\{u_k\}$ such that ${\mathcal{J}}_{\lambda,\mu}(u_k)\rightarrow +\infty$ as $k\rightarrow +\infty$.
			\item[(b)] for every $\lambda>0$, $\mu\in \mathbb{R}$, the concave-convex problem \eqref{problem2} has a sequence of weak solutions $\{v_k\}$ such that ${\mathcal{J}}_{\lambda,\mu}(v_k)\rightarrow 0^-$ as $k\rightarrow +\infty$.
		\end{itemize}
	\end{Theorem}
 
	Before ending this section, we present an application of our previous results. We introduce the following explicit nonlinearities. For problem~\eqref {problem} or \eqref {problem2}, these nonlinearities ensure the existence of the aforementioned weak solutions, including one positive energy solution, solutions with both positive and negative energy, and infinitely many solutions with positive and/or negative energy.
	\begin{Theorem}\label{Thm: application} 
		Let $s=\frac{3}{2}$. Consider the superlinear problem~\eqref{problem} or the concave-convex problem~\eqref{problem2} driven by the higher order fractional Laplacian $(-\Delta)^{3/2}$. Suppose $n=4$ and $f= \partial_{u}F$. The following statements are true.
		\begin{itemize}
			\item [1.] If $f$ is a continuous sectional function defined by
			\[   \quad\quad\quad f(x,u)= f_1(x,u)= 
				\begin{cases}
					2u\ln (|u|+1)+ \frac{u^3}{|u|(|u|+1)} & |u|\geqslant T_1, ~ -a\leqslant u \leqslant b, \\
					\text{continuous function} & -T_1< u <-a, ~b< u<T_1,
				\end{cases}  \]
	        where $T_1< T_0$ and $0<a<b$. It is easy to check the function $f_1$ does not satisfy either condition (S) or (AR) condition.	        
			However, due to Theorem~\ref{Mountain Pass sol theorem}, there is a Mountain Pass solution of the superlinear problem~\eqref{problem}. \smallskip		
			\item [2.] If $f(x,u)= f_2(x,u)= \lambda|u|^{-1/2}u+ f_1(x,u)$, where $\lambda>0$ small enough, then, due to Theorem~\ref{negative energy sol theorem}, there are both a positive energy solution and a negative energy solution of the concave-convex problem~\eqref{problem2}.
			\smallskip
			\item [3.] If $F(x,u)= F_3(x,u)= u^2 \ln (|u| +1)$, which does not satisfy (AR) condition, then, due to Theorem~\ref{infinitely many sol Theorem}, there are infinitely many positive energy solutions of the superlinear problem~\eqref{problem}.
			\smallskip
			\item [4.] If $F(x,u)= F_4(x,u)= \lambda |u|^{3/2} + u^2 \ln (|u| +1)$, where $\lambda>0$, then, due to Theorem~\ref{infinite positive and negative energy sol theorem},  there are infinitely many solutions with both positive energy and  negative energy of the concave-convex problem~\eqref{problem2}.
			\smallskip
		    \item [5.] Let $f(x,u)= f_{5}(x,u) = \lambda|u|^{p-2}u + \mu|u|^{q-2}u$, where $1<p<2<q<2^{*}_{s}$, then 
		    \begin{itemize}
		    \item [(i)] when $ \mu>0, \lambda\in \mathbb{R}$, due to item (a) of Theorem~\ref{infinite positive and negative energy sol theorem-2}, there are  infinitely many weak solutions with positive energy of the concave-convex problem~\eqref{problem2};
		    \item [(ii)] when $\lambda>0, \mu\in \mathbb{R}$, due to item (b) of Theorem~\ref{infinite positive and negative energy sol theorem-2}, there are  infinitely many weak solutions with negative energy of the concave-convex problem~\eqref{problem2}.
		    \end{itemize}
		\end{itemize}
	\end{Theorem}	

	The paper is organized as follows. In Section~\ref{sec:preliminary}, we introduce some preliminary facts and the variational framework for the elliptic  problem driven by the higher order fractional Laplacian with Dirichlet boundary condition.
	 In Section~\ref{sec MPS},  for  the superlinear problem~\eqref{problem}, we  obtain a non-trivial weak solution of positive energy using Mountain Pass Theorem; for the concave-convex problem~\eqref{problem2}, besides the Mountain Pass solution, we get  another negative energy solution  by Ekeland's variational principle. 
	 Due to Fountain Theorem and Dual Fountain Theorem, infinitely many solutions  are established under symmetry condition in Section \ref{sec: infinite sol}.	
	
	\section{Preliminaries}\setcounter{equation}{0}	\label{sec:preliminary}
	Let $s = 1+ \sigma$, where $s \in (1,2), n>2s$. Also, let 	 $\Omega$ be an open bounded set in $\mathbb{R}^n$ with $C^{1}$ boundary. One could know from \cite{MR3856149} that
	\begin{equation*}
		\mathcal{H}_{0}^{s}:=\left\{u \in H^{s}\left(\mathbb{R}^{n}\right): u \equiv 0 \text { in } \mathbb{R}^{n} \backslash \Omega \right\}
	\end{equation*}
	equipped with the norm
	\begin{equation}\label{Xnorm}
		\|u\|_{H^s(\mathbb{R}^n)}=\bigg(\|u\|^{2}_{L^2(\mathbb{R}^{n})}+2\|Du\|^{2}_{L^2(\mathbb{R}^{n})}+\int_{\mathbb{R}^{2n}}\frac{|Du(x)-Du(y)|^2}{|x-y|^{n+2\sigma}}\, dxdy\bigg)^{\frac{1}{2}}
	\end{equation}
         is a suitable Hilbert space to build a variational structure admitted by the higher order fractional Laplacian. Note that $\mathcal{H}_{0}^{s}$ is different from the usual Sobolev spaces $H^{s}(\mathbb{R}^{n})$ and $H^{s}(\Omega)$.

   We can directly obtain the following lemma from \cite[Theorem 6.5]{DPV12}.
	\begin{Lemma}\label{Lemma: equivalent norm} Let $u\in \mathcal{H}_{0}^{s},$ then the norm in \eqref{Xnorm} is equivalent to the 
			 following norm:
		\begin{equation}\label{equivalent norm}
			\|u\|_{\mathcal{H}_{0}^{s}} :=\left(\int_{\mathbb{R}^{2n}} \frac{|D u(x)-D u(y)|^{2}}{|x-y|^{n+2\sigma}} d x d y\right)^{\frac{1}{2}}.
		\end{equation}
	\end{Lemma}

	Hence, in the present paper, we choose the Hilbert space  $(\mathcal{H}_{0}^{s},\|\cdot\|_{\mathcal{H}_{0}^{s}})$ with scalar product
	\[\langle u, v\rangle_{\mathcal{H}_{0}^{s}}=\int_{\mathbb{R}^{2n}}\frac{(Du(x)-Du(y))(Dv(x)-Dv(y))}{|x-y|^{n+2\sigma}}  d x d y. \]
    
	We say that $u\in \mathcal{H}_{0}^{s}$ is a weak solution of problem \eqref{problem}, if $u$ satisfies
	\begin{equation}\label{ definition of weak problem}
		\int_{\mathbb{R}^{2n}} \frac{(D u(x)-D u(y))(D \varphi(x)-D \varphi(y))}{|x-y|^{n+2\sigma}} d x d y=\int_{\Omega} f(x, u) \varphi d x,
	\end{equation}
	for any $ \varphi \in \mathcal{H}_{0}^{s}$.
	
	The weak solution of problem~\eqref{problem} can be found as a critical point of the functional $\mathcal{J}:\mathcal{H}_{0}^{s}\rightarrow \mathbb{R}$ defined by
	\begin{equation}\label{J}
		\mathcal{J}(u)=\frac{1}{2} \|u\|_{\mathcal{H}_{0}^{s}}^{2}-\int_{\Omega} F(x, u(x)) d x 
	\end{equation}
	where $$ F(x,u)=\int_{0}^{u} f(x,\tau)\, d\tau .$$ \par 
	Indeed, the functional $\mathcal{J}$ is Fr\'echet differentiable in $\mathcal{H}_{0}^{s}$ and 
	\[\left\langle\mathcal{J}^{\prime}(u), \varphi\right\rangle= \langle u, \varphi \rangle_{\mathcal{H}_{0}^{s}}-\int_{\Omega} f(x, u(x)) \varphi(x) d x\]
	for any $\varphi \in \mathcal{H}_{0}^{s}$.\par 
	
	To well investigate the variational properties of the functional $\mathcal{J}$, we need  the following two embedding lemmata of  $\mathcal{H}_{0}^{s}$. 
	\begin{Lemma}\label{Lemma:X is embedded in fractional Sobolev space}
		If $v \in \mathcal{H}_{0}^{s}$, then  
		\[
		\tilde{C}\|v\|_{L^q(\Omega)}\leqslant  \|v\|_{H^{s}(\Omega)} \leqslant  \|v\|_{H^{s}\left(\mathbb{R}^{n}\right)}
		 \leqslant C \|v\|_{\mathcal{H}_0^s}, \quad \text{ for  any } q \in [1, 2_{s}^{*}],
		\]
		where $\tilde{C}$ and $ C$ are some positive constants. Consequently, the space  $\mathcal{H}_{0}^{s}$ is continuously embedded in $L^{q}(\Omega)$, $ q \in[1,2_{s}^{*}] $.
	\end{Lemma}
    
{	\begin{Lemma}\label{X10 compact embedded in Lq}
		Let $v_{j}$ be a bounded sequence in $\mathcal{H}_{0}^{s}$. Then, there exists $v_{\infty}$  such that, up to a subsequence,
		\[v_{j} \rightarrow v_{\infty} \text { in } L^{q}\left(\mathbb{R}^{n}\right), \text{ as } j \rightarrow+\infty, \text{ for any } q \in\left[1,2^{*}_{s}\right) .\]
	\end{Lemma}
	\begin{proof}
	Let $2^*_\sigma = \frac{2n}{n-2\sigma}$. Since $\{v_j\}$ is bounded in $\mathcal{H}_0^s$,  $\{Dv_{j}\}$ is bounded in $H^{\sigma}(\Omega)$. Via the embedding theorem  \cite{DPV12}, $\{Dv_{j} \}$ is bounded in $L^{2^*_\sigma}(\Omega)$. Consequently, $\{v_j\}$ is bounded in $W^{1,2^*_\sigma}(\Omega)$ due to the Poincar\'e inequality on $\Omega$.
  
  As $2^*_\sigma > 2$ and $\Omega$ is a bounded domain with $C^1$ boundary, we have  the compact embedding  (e.g. in \cite{MR2597943})
  \[W^{1,2^*_\sigma}(\Omega) \hookrightarrow \hookrightarrow L^q(\Omega) \quad \text{ for } q \in [1,2^*_s), \] 
  which provides a subsequence of $\{v_j\}$ (still denoted by $\{v_j\}$) such that
  $v_j \to v_\infty$ strongly in $L^q(\Omega)$. 
  
  Since $v_{j}$ vanishes outside $\Omega$, one can define $v_{\infty}:=0$ in $\mathbb{R}^{n} \setminus \Omega$ and we have $v_\infty\in L^q(\mathbb{R}^n)$.
  Thus, the convergence $v_j \to v_\infty$ is in $L^q(\mathbb{R}^n)$, for any $q \in [1,2^*_s)$. 
	\end{proof}
	}

	We would remark that, if we replace $(-\Delta)^{s}$  with $s \in (1,2)$ by a slightly generalized nonlocal operator
	$$\widetilde{\mathscr{L}}_{K}:=(-\Delta)(\mathscr{L}_{K}) $$
	where the operator $\mathscr{L}_{K}$ and the kernel $K$ are defined as in \cite{SV12}, the associated compact embedding theorem could be proved analogously.
		
\smallskip
In the following, we provide several spectral properties of the higher order fractional Laplacian operator $(-\Delta)^s$, which will be used later to find infinitely many solutions. 
\begin{Lemma}\label{Prop: X0 separable}
	The weak eigenvalue problem associated to $(-\Delta)^s$ for $s\in (1,2)$
  \begin{equation}\label{eigen problem0}
	\left\{\begin{array}{l}
		\langle u, \varphi \rangle_{\mathcal{H}_{0}^{s}}=\lambda \int_{\Omega} u(x) \varphi(x) d x, \quad \forall \varphi \in \mathcal{H}_{0}^{s} \\
		u \in \mathcal{H}_{0}^{s}
	\end{array}\right.
  \end{equation}
possesses a divergent sequence of positive eigenvalues $\left\{\lambda_{k}\right\}_{k \in \mathbb{N}}$ with 
\[0<\lambda_{1}\leqslant\lambda_{2} \leqslant \cdots \leqslant \lambda_{k} \leqslant \lambda_{k+1} \leqslant \cdots\]
whose corresponding eigenfunctions $\{e_{k}\}_{k \in \mathbb{N}}$ can be chosen in such a way that this sequence is an orthonormal basis of $L^{2}(\Omega)$ and an orthogonal basis of $\mathcal{H}_{0}^{s}$. 
\end{Lemma}
The proof of Lemma~\ref{Prop: X0 separable} is deferred to Appendix with more details.

In conclusion, we have the following proposition by Lemma~\ref{Prop: X0 separable}. 
\begin{Proposition}
	$\left(\mathcal{H}_{0}^{s},\|\cdot\|_{\mathcal{H}_{0}^{s}}\right)$ is a separable Hilbert space.
\end{Proposition}
	
	\section{Existence of two solutions without symmetry condition}\label{sec MPS}\setcounter{equation}{0}
	
         \subsection{Mountain Pass solution without (AR) condition}
	In this section, we obtain a positive energy solution to superlinear problem~\eqref{problem} by the Mountain Pass Theorem \cite{AR73}.  
	We start by proving the necessary geometric features of the functional $\mathcal{J}$.
	\begin{Proposition}\label{MP geometry}
		Let $f: \overline{\Omega} \times \mathbb{R} \rightarrow \mathbb{R}$ be a Carath\'eodory function. 
		\begin{itemize}
		\item [$(a)$] If $f$ satisfies (H1) and (H2),  then there exist $\rho>0$ and $\beta>0$ such that for any $u \in \mathcal{H}_{0}^{s}$ with $\|u\|_{\mathcal{H}_{0}^{s}}=\rho$, we have $\mathcal{J}(u) \geqslant \beta$.
		\item [$(b)$] If $f$ satisfies (H3), then there exists $e \in \mathcal{H}_{0}^{s}$ such that $\|e\|_{\mathcal{H}_{0}^{s}}>\rho$ and $\mathcal{J}(e)<\beta$.
		\end{itemize}
	\end{Proposition}
	\begin{proof}
		$(a).$ (H1) and (H2) imply that,  
		for any $\varepsilon>0$ there exists $\delta=\delta(\varepsilon)$ such that a.e. $x \in \Omega$ and for any $t \in \mathbb{R}$
		\begin{equation}\label{estimate of F}
			|F(x, t)| \leqslant \varepsilon|t|^{2}+\delta(\varepsilon)|t|^{q}.
		\end{equation} 
	    Let $u\in \mathcal{H}_{0}^{s}$. By \eqref{estimate of F} and Lemma \ref{Lemma:X is embedded in fractional Sobolev space}, we get that for any $\varepsilon>0$,
		\begin{equation}\label{7}
			\begin{aligned} \mathcal{J}(u)  =&\frac{1}{2} \|u\|_{\mathcal{H}_{0}^{s}}^{2}-\int_{\Omega} F(x, u(x)) d x
			\geqslant \frac{1}{2} \|u\|_{\mathcal{H}_{0}^{s}}-\varepsilon|u|_2^{2}-\delta(\varepsilon)|u|_q^{q} \\  
			\geqslant &\frac{1}{2} \|u\|_{\mathcal{H}_{0}^{s}}^{2} -\varepsilon|\Omega|^{\left(2^{*}_{s}-2\right) / 2^{*}_{s}}|u|_{2^{*}_{s}}^{2}-|\Omega|^{\left(2^{*}_{s}-q\right) / 2^{*}_{s}} \delta(\varepsilon)|u|_{2^{*}_{s}}^{q}\\
		    \geqslant & \frac{1}{2} \|u\|_{\mathcal{H}_{0}^{s}}^{2}  -\varepsilon |\Omega|^{\left(2^{*}_{s}-2\right) / 2^{*}_{s}}C \|u\|_{\mathcal{H}_{0}^{s}}^{2} -\delta(\varepsilon) |\Omega|^{\left(2^{*}_{s}-q\right) / 2^{*}_{s}}C \|u\|_{\mathcal{H}_{0}^{s}}^{q}\\
		    \geqslant &\left(\frac{1}{2}-\varepsilon C |\Omega|^{\left(2^{*}_{s}-2\right) /2^{*}_{s}}\right) \|u\|_{\mathcal{H}_{0}^{s}}^{2} -\delta(\varepsilon)C |\Omega|^{\left(2^{*}_{s}-q\right) / 2^{*}_{s}}\|u\|_{\mathcal{H}_{0}^{s}}^{q} .\end{aligned}
		\end{equation}
		 Here and in the sequel, we denote by $|\Omega |$ the Lebesgue measure of $\Omega $.\par 
		Choosing $\varepsilon>0$ such that $\frac{1}{2}-\varepsilon C |\Omega|^{\left(2^{*}_{s}-2\right) / 2^{*}_{s}}>0$, it easily follows that
		\[\begin{aligned}
			\mathcal{J}(u) \geqslant \alpha\|u\|_{\mathcal{H}_{0}^{s}}^{2}\left(1-\kappa\|u\|_{\mathcal{H}_{0}^{s}}^{q-2}\right)
		\end{aligned}\]
		for suitable positive constants $\alpha$ and $\kappa$.\par
		Now, let $u \in \mathcal{H}_{0}^{s}$ be such that $\|u\|_{\mathcal{H}_{0}^{s}}=\rho>0$. By assumption $q>2$, one can choose $\rho>0$ sufficiently small such that $1-\kappa \rho^{q-2}>0$, and so
		\[\inf _{u \in \mathcal{H}_{0}^{s},~\|u\|_{\mathcal{H}_{0}^{s}}=\rho} \mathcal{J}(u) \geqslant \alpha \rho^{2}\left(1-\kappa \rho^{q-2}\right)=: \beta>0.\]\par 
		$(b).$ (H3) implies that,  for all $M>0$, there exists $C_{M}>0$ such that
		\begin{equation}\label{H3}
			F(x, t) \geqslant M t^{2}-C_{M}, \quad  \forall  x \in \Omega,  t \in \mathbb{R} .
		\end{equation}\par
		We fix $\phi  \in \mathcal{H}_{0}^{s}$ such that $\|\phi \|_{\mathcal{H}_{0}^{s}}=1$.
		Let $t\in \mathbb{R}$. We have
		\[\begin{aligned}
			\mathcal{J}(t \phi)
			& = \frac{1}{2}\|t\phi\|_{\mathcal{H}_{0}^{s}}^{2}-\int_{\Omega}F(x,t\phi) dx\\
			& \leqslant \frac{t^{2}}{2}-\int_{\Omega} M t^{2} \phi^{2} d x+\int_{\Omega} C_{M} d x 
			& =t^{2}\left(\frac{1}{2}-M |\phi|_2^2\right)+C_{M}|\Omega|.
		\end{aligned}\]\par 
		Let $M= \frac{1}{2|\phi |_2^2}+1$. Passing to the limit as $t \rightarrow +\infty$, we get that $\mathcal{J}(t \phi ) \rightarrow-\infty$, so that the assertion follows by taking $e=T \phi $, with $T>0$ sufficiently large.
	\end{proof}
	We now check the validity of the Palais-Smale condition, i.e., every Palais-Smale sequence of $\mathcal{J}$ has a convergent subsequence in $\mathcal{H}_{0}^{s}$.

	\begin{Proposition}\label{PS bounded sequence}
		Let $f: \overline{\Omega} \times \mathbb{R} \rightarrow \mathbb{R}$ be a Carath\'eodory function satisfying (H3) and (H4). Let $\{u_{j}\}$ be a Palais-Smale sequence of $\mathcal{J}$ in $\mathcal{H}_{0}^{s}$, i.e.,
		\begin{equation}\label{T' convergent to 0}
			\mathcal{J}(u_{j}) \rightarrow c, c\in \mathbb{R} \quad \text{ and } \quad
			\mathcal{J}^{\prime}(u_{j}) \rightarrow 0 \text{ as } j \rightarrow+\infty.
		\end{equation}
	 Then $\{u_j\}$ is bounded in $\mathcal{H}_{0}^{s}$.
	\end{Proposition}
	\begin{remark}
	We prove Proposition~\ref{PS bounded sequence}  below by assuming a slightly weaker condition (H4)* than (H4), see \cite[Remark 1.1]{MIYAGAKI08}. \par
		\textbf{(H4)*} Denote $H(x, s)=s f(x, s)-2 F(x, s)$. There exists $C_{*}>0$ such that
		\begin{equation}\label{H4}
			H(x, t) \leqslant H(x, s)+C_{*}
		\end{equation}
		for all $0<t<s$ or $s<t<0$, $\forall x \in \Omega$. \par  
	\end{remark}

	\begin{proof}
		Take any Palais-Smale sequence $\{u_j\}$ of $\mathcal{J}$ in $\mathcal{H}_{0}^{s}$.  
		We suppose, by contradiction, that up to a subsequence, still denoted by $u_{j}$,
		\begin{equation}\label{unbo u}
		\|u_{j}\|_{\mathcal{H}_{0}^{s}}  \rightarrow+\infty \text { as } j \rightarrow+\infty .
		\end{equation}
	
		Set $\omega_{j}:=\frac{u_{j}}{\left\|u_{j}\right\|_{\mathcal{H}_{0}^{s}}}$. Then
		$\|\omega_{j}\|_{\mathcal{H}_{0}^{s}}=1 .$
		Since $\omega_{j}$ is bounded in $\mathcal{H}_{0}^{s}$ and $\mathcal{H}_{0}^{s}$ is a Hilbert space,
		we may assume that there exists $\omega \in \mathcal{H}_{0}^{s}$ such that 
		\begin{align*}
			\omega_{j} &\rightharpoonup \omega, \quad \text{weakly in } \mathcal{H}_{0}^{s}, \\
			\omega_{j} &\rightarrow \omega, \quad \text{strongly in } L^{2}(\Omega), \\
			\omega_{j}(x) &\rightarrow \omega(x), \quad \text{a.e. in } \Omega.
		\end{align*}\par 
		We denote $\Omega^{*}:=\{x \in \Omega, \omega(x) \neq 0\}$. If $\Omega^{*} \neq \emptyset$, then for $x \in \Omega^{*}$, we can deduce that $|u_{j}(x)| \rightarrow+\infty$ as $j \rightarrow+\infty$ from \eqref{unbo u}. By (H3), we have
		\[\lim \limits_{j \rightarrow+\infty} \frac{F(x, u_{j}(x))}{(u_{j}(x))^{2}}(\omega_{j}(x))^{2}=+\infty .\]\par
		The Fatou's Lemma and the definition of $\omega_{j}$ imply
		\[\begin{aligned}
			&\quad \int_{\Omega} \lim \limits_{j \rightarrow+\infty} \frac{F(x, u_{j}(x))}{(u_{j}(x))^{2}}(\omega_{j}(x))^{2} d x =\int_{\Omega} \lim \limits_{j \rightarrow+\infty} \frac{F(x, u_{j}(x))}{(u_{j}(x))^{2}} \frac{(u_{j}(x))^{2}}{\|u_{j}\|_{\mathcal{H}_{0}^{s}}^{2}} d x \\
			& \leqslant \liminf _{j \rightarrow+\infty} \frac{1}{\|u_{j}\|_{\mathcal{H}_{0}^{s}}^{2}} \int_{\Omega} F(x, u_{j}(x)) d x 
			 =\lim \limits_{j \rightarrow \infty} \frac{1}{\|u_{j}\|_{\mathcal{H}_{0}^{s}}^{2}}\left(\frac{1}{2}\|u_{j}\|_{\mathcal{H}_{0}^{s}}^{2}-\mathcal{J}(u_{j})\right)  =\frac{1}{2} .
		\end{aligned}\]\par
		If $\Omega^{*}$ has positive measure, the integration above won't be a finite value. Hence $\Omega^{*}$ has zero measure. Consequently, $\omega(x) \equiv 0$ a.e. in $\Omega$.\par
		Consequently, we take $t_{j} \in[0,1]$ such that
		\[\mathcal{J}(t_{j} u_{j})=\max _{t \in[0,1]} \mathcal{J}_{}(t u_{j}).\]\par 
		Then we can deduce that
		\begin{equation}\label{5}
			\left.\frac{d}{d t} \mathcal{J}(t u_{j})\right|_{t=t_{j}}=t_{j}\|u_{j}\|_{\mathcal{H}_{0}^{s}}^{2}-\int_{\Omega} f(x, t_{j} u_{j}) \cdot u_{j} d x=0 .
		\end{equation}\par 
		Since
		\[\langle\mathcal{J}^{\prime}(t_{j} u_{j}), t_{j} u_{j}\rangle=t_{j}^{2}\|u_{j}\|_{\mathcal{H}_{0}^{s}}^{2}-\int_{\Omega} f(x, t_{j} u_{j}) \cdot t_{j} u_{j} d x,\]
		together with \eqref{5}, it follows that
		\[\langle\mathcal{J}^{\prime}(t_{j} u_{j}), t_{j} u_{j}\rangle=\left.t_{j} \cdot \frac{d}{d t} \mathcal{J}(t u_{j})\right|_{t=t_{j}}=0 .\]\par
		
		Hence, by \eqref{H4}, we obtain
		\begin{equation}\label{6}
			\begin{aligned}
				2 \mathcal{J}(t u_{j}) & \leqslant 2 \mathcal{J}_{}(t_{j} u_{j})-\langle\mathcal{J}^{\prime}(t_{j} u_{j}), t_{j} u_{j}\rangle \\
				& =\int_{\Omega}\left(t_{j} u_{j} \cdot f(x, t_{j} u_{j})-2 F(x, t_{j} u_{j})\right) d x\\
				&\leqslant  \int_{\Omega}\left(u_{j} \cdot f(x, u_{j})-2 F(x, u_{j})+C_{*}\right) d x \\
				&=  2 \mathcal{J}(u_{j})-\langle\mathcal{J}^{\prime}(u_{j}), u_{j}\rangle+C_{*}|\Omega| \rightarrow 2 c+C_{*}|\Omega| .
			\end{aligned}
		\end{equation}\par 
		On the other hand, for all $k>0$,
		\[2 \mathcal{J}_{}(k \omega_{j})=k^{2}-2 \int_{\Omega} F_{}(x, k \omega_{j}) d x=k^{2}+o(1),\quad \text{ as } j\rightarrow+\infty,\]
		which contradicts \eqref{6} for $k$ and $j$ large enough. So every Palais-Smale sequence of $\mathcal{J}$ is bounded in $\mathcal{H}_{0}^{s}$.
	\end{proof}

	\begin{Proposition}\label{PSc convergent sequence}
		Let $f: \overline{\Omega} \times \mathbb{R} \rightarrow \mathbb{R}$ be a continuous function satisfying (H1). Let $u_{j}$ be a bounded sequence in $\mathcal{H}_{0}^{s}$ such that 
		\begin{equation}\label{T' convergent to 0-2}
			\left\| \mathcal{J}^{\prime}(u_{j}) \right\|:= \sup\left\{ | \langle \mathcal{J}^{\prime}(u_{j}),\varphi \rangle |: \varphi \in  \mathcal{H}_0^s, ~\| \varphi \|_{\mathcal{H}_0^s}=1 \right\} \rightarrow 0
			\text{ as } j \rightarrow + \infty. 
		\end{equation} 
	    Then there exists $u_{\infty} \in \mathcal{H}_{0}^{s}$ such that, up to a subsequence, 
	    \[\|u_{j}-u_{\infty}\|_{\mathcal{H}_{0}^{s}} \rightarrow 0\text{ as }j \rightarrow+\infty.\]
	\end{Proposition}

	\begin{proof}
		Since $u_{j}$ is bounded in $\mathcal{H}_{0}^{s}$ and $\mathcal{H}_{0}^{s}$ is a Hilbert space, up to a subsequence, still denoted by $u_{j}$, there exists $u_{\infty} \in \mathcal{H}_{0}^{s}$ such that
		\begin{equation}\label{weak covergence}
			\langle u_{j}, \varphi \rangle_{\mathcal{H}_{0}^{s}} \rightarrow \langle u_{\infty}, \varphi \rangle_{\mathcal{H}_{0}^{s}},
		\text { for any } \varphi \in \mathcal{H}_{0}^{s} \text{ as } j \rightarrow+\infty.
		\end{equation}
		{ Moreover, by Lemma \ref{X10 compact embedded in Lq}, up to a subsequence,
		\begin{equation}
			u_{j} \rightarrow u_{\infty}  \text { in } L^{q}\left(\mathbb{R}^{n}\right), \ \text{ as } j \rightarrow+\infty. \label{Lqconvergence}
		\end{equation}	
		From \cite[Theorem 4.9]{Brezisbook} and \eqref{Lqconvergence}, one can obtain a subsequence $u_{j_i}$ and  $\ell \in L^{q}\left(\mathbb{R}^{n}\right)$ such that
		\[
		\begin{split}
					&u_{j_i} \rightarrow u_{\infty} 
			\text { a.e. in } \mathbb{R}^{n}, \ \text{ as } j \rightarrow+\infty, \\
		& |u_{j_i}(x)| \leqslant \ell(x) \quad \text { a.e. in } \mathbb{R}^{n}, \text { for any } j \in \mathbb{N}.
		\end{split}
		\]
		Without abuse of notation, we still denote the subsequence $u_{j_i}$ by $u_j$ below.}
		
		By (H1), the fact that the map $t \mapsto f(\cdot, t)$ is continuous in $t \in \mathbb{R}$ and the Dominated Convergence Theorem, we get
		\begin{equation}\label{nonlinear j coenvergent}
			\int_{\Omega} f(x, u_{j}(x)) u_{j}(x) d x \rightarrow \int_{\Omega} f\left(x, u_{\infty}(x)\right) u_{\infty}(x) d x, \text{ as } j \rightarrow +\infty,
		\end{equation}
		and
		\begin{equation}\label{nonlinear infinity coenvergent}
			\int_{\Omega} f(x, u_{j}(x)) u_{\infty}(x) d x \rightarrow \int_{\Omega} f\left(x, u_{\infty}(x)\right) u_{\infty}(x) d x, \text{ as } j \rightarrow+\infty.
		\end{equation} 
		
		From \eqref{T' convergent to 0-2}, it follows that 
		\[\langle\mathcal{J}^{\prime}(u_{j}), u_{j}\rangle\rightarrow 0 , \text{ and  } \ \langle\mathcal{J}^{\prime}(u_{j}), u_{\infty}\rangle\rightarrow 0  \quad \text{ as }  j \rightarrow  +\infty. \]
		Therefore,
		\begin{equation}\label{J u 0}
			0 \leftarrow\langle\mathcal{J}^{\prime}(u_{j}), u_{j}\rangle= \|u_{j}\|_{\mathcal{H}_{0}^{s}}^{2}-\int_{\Omega} f(x, u_{j}(x)) u_{j}(x) d x.
		\end{equation}\par 
		By \eqref{nonlinear j coenvergent} and \eqref{J u 0} we deduce that
		\begin{equation}\label{3}
			\|u_{j}\|_{\mathcal{H}_{0}^{s}}^{2} \rightarrow \int_{\Omega} f\left(x, u_{\infty}(x)\right) u_{\infty}(x) d x \text{ as } j \rightarrow+\infty.
		\end{equation}
		 Furthermore,
		\[ 0 \leftarrow\langle\mathcal{J}^{\prime}(u_{j}), u_{\infty}\rangle= \langle u_{j}, u_{\infty}\rangle_{\mathcal{H}_{0}^{s}} -\int_{\Omega} f(x, u_{j}(x)) u_{\infty}(x) d x. \]\par
		So, by \eqref{nonlinear infinity coenvergent} we deduce that
		\begin{equation}\label{1}
		    \langle u_{j}, u_{\infty}\rangle_{\mathcal{H}_{0}^{s}}\rightarrow\int_{\Omega} f\left(x, u_{\infty}\right) u_{\infty}dx \text{ as } j \rightarrow+\infty.
		\end{equation}
		Let $\varphi=u_{\infty}$ in \eqref{weak covergence} and by \eqref{1}, we get
		\begin{equation}\label{2}
			 \|u_{\infty}\|_{\mathcal{H}_{0}^{s}}^{2}=\int_{\Omega} f\left(x, u_{\infty}(x)\right) u_{\infty}(x)dx.
		\end{equation}\par 
		By \eqref{3} and \eqref{2}, we get
		\begin{equation}\label{4}
			\|u_j\|_{\mathcal{H}_{0}^{s}}^2 \rightarrow \|u_{\infty}\|_{\mathcal{H}_{0}^{s}}^2.
		\end{equation}\par 
		Finally, we have that
		\[\begin{aligned}
			\|u_{j}-u_{\infty}\|_{\mathcal{H}_{0}^{s}}^{2}&=\|u_{j}\|_{\mathcal{H}_{0}^{s}}^{2}+\left\|u_{\infty}\right\|_{\mathcal{H}_{0}^{s}}^{2}-2 
			\langle u_{j}, u_{\infty}\rangle_{\mathcal{H}_{0}^{s}} \rightarrow 0 \text{ as } j \rightarrow+\infty,
		\end{aligned}\]
		 thanks to \eqref{weak covergence} and \eqref{4}. 
	\end{proof}
	
	\begin{proof}[\textbf{Proof of Theorem~\ref{Mountain Pass sol theorem}}] Due to Propositions~\ref{MP geometry}-\ref{PSc convergent sequence}, the Mountain Pass Theorem gives that there exists a critical point $u \in \mathcal{H}_{0}^{s}$ of $\mathcal{J}$ which is actually the non-trivial weak solution of \eqref{problem}. Moreover, we have
		\[\mathcal{J}(u) \geqslant \beta>0=\mathcal{J}(0)\]
	where $\beta$ is given in Proposition~\ref{MP geometry}.
	\end{proof}
	
\subsection{Negative energy solution by Ekeland's variational principle}\label{sec a negative energy sol}\setcounter{equation}{0}
	In this section, we consider the concave-convex nonlinear problem
	\begin{equation*}
		\left\{
		\begin{array}{ll}
			(-\Delta)^{s} u=\lambda|u|^{p-2}u+g(x,u) & \text{ in }\Omega, \\
			u=0 &  \text{ in } \mathbb{R}^n\setminus\Omega \\
		\end{array}
		\right.
	\end{equation*}
	where $1<p<2$, $\lambda \geqslant 0$ is a parameter and $g(x,t)$ is a function on $\overline{\Omega}\times\mathbb{R}$. \par 

	\begin{Lemma}\label{negative energy sol lemma}
		Assume that $g: \overline{\Omega} \times \mathbb{R} \rightarrow \mathbb{R}$ is a Carath\'eodory function.
		\begin{itemize}
		\item [$(a)$] If $g$ also satisfies (H3), then $\mathcal{J_{\lambda}}$ is unbounded from below.
		\item [$(b)$] If $g$ also satisfies (H1) and (H2), then for $\lambda>0$ small enough, there exist $\rho, R>0$ such that $\mathcal{J}_{\lambda}(u) \geqslant R$, if $\|u\|_{\mathcal{H}_{0}^{s}}=\rho$.
		\item [$(c)$] If $g\in C(\overline{\Omega} \times \mathbb{R}, \mathbb{R})$ also satisfies (H1), (H3) and (H4), then $\mathcal{J_{\lambda}}$ satisfies the Palais-Smale condition.
		\end{itemize}
	\end{Lemma}

	\begin{proof}
		$(a)$ The proof is similar to the proof of item $(b)$ in Proposition~\ref {MP geometry}.
		
		$(b)$ Combining \eqref{estimate of F} and Lemma \ref{Lemma:X is embedded in fractional Sobolev space}, we have
		\[\begin{aligned}
			\mathcal{J}_{\lambda}(u) & = \frac{1}{2}\|u\|_{\mathcal{H}_{0}^{s}}^{2}-\frac{\lambda}{p}|u |_{p}^{p}-\int_{\Omega}G(x,u)dx 
			 \geqslant \frac{1}{2}\|u\|_{\mathcal{H}_{0}^{s}}^{2}-\frac{\lambda}{p}|u |_{p}^{p}-\varepsilon|u|_{2}^{2}-\delta(\varepsilon)|u|_{q}^{q} \\
			& \geqslant \frac{1}{2}\|u\|_{\mathcal{H}_{0}^{s}}^{2}-\frac{\lambda}{p}|\Omega|^{\left(2_{s}^{*}-p\right) / 2_{s}^{*}}|u|_{2_{s}^{*}}^{p}-\varepsilon|\Omega|^{\left(2_{s}^{*}-2\right) / 2_{s}^{*}}|u|_{2_{s}^{*}}^{2}-\delta(\varepsilon)|\Omega|^{\left(2_{s}^{*}-q\right) / 2_{s}^{*}}|u|_{2_{s}^{*}}^{q} \\
			& \geqslant \frac{1}{2}\|u\|_{\mathcal{H}_{0}^{s}}^{2}-\lambda K||u||_{\mathcal{H}_{0}^{s}}^{p}-\varepsilon C_0||u||_{\mathcal{H}_{0}^{s}}^{2}-C_q(\varepsilon)||u||_{\mathcal{H}_{0}^{s}}^{q} \\
			& =||u||_{\mathcal{H}_{0}^{s}}^{2}\left(A-\lambda K\|u\|_{\mathcal{H}_{0}^{s}}^{p-2}-C_q(\varepsilon)\|u\|_{\mathcal{H}_{0}^{s}}^{q-2}\right)
		\end{aligned}\]
		where $K, C_q, C_0$ are positive constants and $A=\frac{1}{2}-\varepsilon C_0$. Taking $\varepsilon>0$ small enough we get that the constant $A>0$. Let
		\[Q(t)=\lambda K t^{p-2}+C_q t^{q-2} .\] 
		Then
		$Q^{\prime}(t)=\lambda K(p-2) t^{p-3}+C_q(q-2) t^{q-3} .$\par 
		Setting
		$Q^{\prime}(t_{0})=0,$
		we know
		\[t_{0}=\left(\frac{\lambda K(2-p)}{C_q(q-2)}\right)^{\frac{1}{q-p}}.\]\par 
		Since $1<p<2<q<2_{s}^{*}$,  $Q(t)$ has a minimum at $t=t_{0}$. Let
		\[\beta=\frac{K(2-p)}{C_q(q-2)}, \quad \bar{p}=\frac{p-2}{q-p}, \quad \bar{q}=\frac{q-2}{q-p} .\]\par 
		Substituting $t_{0}$ in $Q(t)$ we have
		\[Q\left(t_{0}\right)<A, \text{ for } 0<\lambda<\lambda^{*},\]
		where $\lambda^{*}=\left(\frac{A}{K \beta^{\bar{p}}+C_q \beta^{\bar{q}}}\right)^{1 / \bar{q}}$. Taking $\rho=t_{0}$ and $R=A-Q(t_{0})$, we prove part $(b)$.\par 
		
		$(c)$ Since $f_{\lambda}(x, u)=\lambda|u|^{p-2} u+g(x, u)$, $1<p<2$ and $g$ satisfies (H1), (H3) and (H4), we know $f_{\lambda}$ satisfies (H1), (H3) and (H4). 
		By Proposition~\ref{PS bounded sequence}, every Palais-Smale sequence of $\mathcal{J}_{\lambda}$ is bounded in $\mathcal{H}_{0}^{s}$ and $(c)$ follows immediately from Proposition~\ref{PSc convergent sequence}. 
	\end{proof}
	\begin{proof}[\textbf{Proof of Theorem~\ref{negative energy sol theorem}}] 
		Using Lemma~\ref{negative energy sol lemma}, the existence of a positive energy solution $u$ follows by an analogous argument as in the proof of Theorem \ref{Mountain Pass sol theorem}. To obtain another negative energy solution $v$, we consider
		\[\overline{B}(\rho)=\left\{u \in \mathcal{H}_{0}^{s},\|u\|_{\mathcal{H}_{0}^{s}} \leqslant \rho\right\}, \quad \partial B(\rho)=\left\{u \in \mathcal{H}_{0}^{s},\|u\|_{\mathcal{H}_{0}^{s}}=\rho\right\} ,\]
		where $\rho$ is given in Lemma \ref{negative energy sol lemma}.
		Then $\overline{B}(\rho)$ is a complete metric space with the distance
		\[\operatorname{dist}(u, v)=\|u-v\|_{\mathcal{H}_{0}^{s}}, \quad \forall u, v \in \overline{B}(\rho) .\]\par 
		By Lemma \ref{negative energy sol lemma}, we know for $0<\lambda<\lambda^{*}$,
		$\left.\mathcal{J}_{\lambda}(u)\right|_{\partial B(\rho)} \geqslant R>0 .$\par
		Moreover, it is easy to see that $\mathcal{J}_{\lambda} \in C^{1}(\overline{B}(\rho), \mathbb{R})$, hence $\mathcal{J}_{\lambda}$ is lower semi-continuous and bounded from below on $\overline{B}(\rho)$. Let
		$c_{1}=\inf_ {u \in \bar{B}(\rho)}\mathcal{J}_{\lambda}(u).$\par
		Taking $\bar{v} \in C_{c}^{\infty}(\Omega)$. From (H2), we know that for any $\varepsilon>0$, there exists $T>0$ such that for $0<t<T,\left|G\left(x, t \bar{v}\right)\right| \leqslant \varepsilon |t\bar{v}|^{2}$. Then, for both $t, \varepsilon>0 $ small enough,
		\[\begin{aligned}
			\mathcal{J}_{\lambda}\left(t \bar{v}\right) & =\frac{t^{2}}{2}\left\|\bar{v}\right\|_{\mathcal{H}_{0}^{s}}^{2}-\frac{\lambda t^{p}}{p} |\bar{v}|_{p}^{p} d x-\int_{\Omega} G\left(x, t \bar{v}\right) d x \\
			& \leqslant \frac{t^{2}}{2}\left\|\bar{v}\right\|_{\mathcal{H}_{0}^{s}}^{2}-\frac{\lambda t^{p}}{p} |\bar{v}|_{p}^{p} d x+\varepsilon t^{2}|\bar{v}|^{2}_{2}  <0
		\end{aligned}\]
	        since $1<p<2$. Hence, $c_{1}<0$.\par 
		By Ekeland's variational principle, there exists a sequence $\{v_{k}\} $ in $\overline{B}(\rho)$ such that
	$ \mathcal{J}_{\lambda}(v_{k}) \rightarrow c_{1} \text{ and } \nabla \mathcal{J}_{\lambda}(v_{k}) \rightarrow 0$ as $ k \rightarrow\infty.$
		
		It is easy to see  $\{v_{k}\}$ is a Palais-Smale sequence of $\mathcal{J}_{\lambda}$ in $\mathcal{H}_{0}^{s}$. Due to Lemma~\ref{negative energy sol lemma}-$(c)$, there exists $v\in \mathcal{H}_{0}^{s}$ such that $\mathcal{J}_{\lambda}(v)=c_{1}<0$ and $\nabla \mathcal{J}_{\lambda}(v)=0.$ Therefore, $v$ is a weak solution of the concave-convex problem \eqref{problem2} and $\mathcal{J}_{\lambda}(v)<0$.
	\end{proof}
\section{Infinitely many solutions under symmetry condition}\label{sec: infinite sol}\setcounter{equation}{0}
     In this section, we give the proofs of the existence of infinitely many solutions to both the superlinear problem~\eqref{problem} and the concave-convex problem~\eqref{problem2}. The strategy we take is to apply the Fountain Theorem \cite{Bartsch93} and Dual Fountain Theorem \cite{BW95} respectively to the functional $\mathcal{J}$ and $\mathcal{J_{\lambda, \mu}}$.
     
     We first introduce some notations. For any $k \in \mathbb{N}$, we define
     \[Y_{k}:=\operatorname{span}\left\{e_{1}, \cdots, e_{k}\right\},\quad 
     Z_{k}:=\overline{\operatorname{span}\left\{e_{k}, e_{k+1}, \cdots\right\}} \]
     where $\left\{e_k\right\}_{k\in \mathbb{N}}$ are defined in Lemma~\ref{Prop: X0 separable}.
     
     Since $Y_{k}$ is finite-dimensional, all norms on $Y_{k}$ are equivalent. Therefore, there exist two positive constants $C_{k, q}$ and $\tilde{C}_{k, q}$, depending on $k$ and $q$, such that 
     \begin{equation}\label{finite dim eq norm}
     	C_{k, q}\|u\|_{\mathcal{H}_{0}^{s}} \leqslant\|u\|_{L^q(\Omega)} \leqslant \tilde{C}_{k, q}\|u\|_{\mathcal{H}_{0}^{s}} \text{ for any } u \in Y_{k}.
     \end{equation}

    By a modification of \cite[Lemma 3.8]{Willembook}, we have the following lemma.
     \begin{Lemma}\label{Lamma Fountain}
     	Let $1 \leqslant q<2_{s}^{*}$ and, for any $k \in \mathbb{N}$, let
     	\[\beta_{k}:=\sup \left\{\|u\|_{L^q(\Omega)}: u \in Z_{k},\|u\|_{\mathcal{H}_{0}^{s}}=1\right\} .\]
     	Then, $\beta_{k} \rightarrow 0$ as $k \rightarrow \infty$.
     \end{Lemma}

	\subsection{Positive energy solutions of superlinear problems} In this subsection we give the proof of Theorem~\ref{infinitely many sol Theorem}. 
	We remark that Theorem~\ref{infinitely many sol Theorem} still holds under the usual assumptions (H1), (AR) and (S).\par 
	
	As is well known, Fountain Theorem provides the existence of an unbounded sequence of critical values for a $C^1$ invariant functional. To use Fountain Theorem to seek critical points, the space $\mathcal{H}_{0}^{s}$ needs to satisfy the following condition. 
	\begin{itemize}
		\item [\textbf{(F1)}] the compact group $G$ acts isometrically on the Banach space $\mathcal{H}_{0}^{s}= \overline{\oplus_{j \in \mathbb{N}} X_{j}}$, the spaces $X_{j}$ are invariant and there exists a finite dimensional space $V$ such that, for every $j \in \mathbb{N}, X_{j} \simeq V$ and the action of $G$ on $V$ is admissible.
	\end{itemize}

    In fact, the space $\mathcal{H}_{0}^{s}$ does indeed meet the above condition by choosing $G:=\mathbb{Z}/2=\{1,-1\}$ as the action group on $\mathcal{H}_{0}^{s}$, $X_{j}:=\mathbb{R}e_{j}$ where $\{e_{j}\}$ are defined as eigenfunctions in Lemma~ \ref{Prop: X0 separable} and $V:=\mathbb{R}$.
    While, by (S), $\mathcal{J} \in C^1(\mathcal{H}_0^s, \mathbb{R})$ is an invariant functional for any action $g\in G$. 
    
    Just as in the general case, when using Fountain Theorem, we need the functional $\mathcal{J}$ to satisfy some geometric structures and compactness condition:
    \begin{itemize}
	    \item [\textbf{(F2)}] for every $k\in \mathbb{N}$, there exists $\rho_{k}>\gamma_{k}>0$ such that
	       \begin{itemize}
	       	\item [(i)] $a_{k}:=\max \left\{\mathcal{J}(u): u \in Y_{k},\|u\|_{\mathcal{H}_{0}^{s}}=\rho_{k}\right\} \leqslant 0$,
	       	\item [(ii)]  $b_{k}:=\inf \left\{\mathcal{J}(u): u \in Z_{k},\|u\|_{\mathcal{H}_{0}^{s}}=\gamma_{k}\right\} \rightarrow +\infty$ as $k \rightarrow +\infty$;
	       \end{itemize}
	    \item [\textbf{(F3)}] $\mathcal{J}$ satisfies the $(PS)_{c}$ condition for every $c>0$.
	\end{itemize}

As shown in Proposition~\ref{PS bounded sequence} and \ref{PSc convergent sequence}, (F3) is easy to obtain. Now we turn to prove that the functional $\mathcal{J}$ does indeed have the above geometric structure (F2).

	\begin{proof}[\bf {Proof of Theorem~\ref{infinitely many sol Theorem}}]
		Based on the above discussion, we already have the spatial conditions of $\mathcal{H}_{0}^{s}$ and compactness condition of the functional $\mathcal{J}$ required to utilize the Fountain Theorem.
		 
		As for the geometric feature (F2) of $\mathcal{J}$, we verify the assumption (ii) firstly. For this, 
		we just need to prove that, for every $k \in \mathbb{N}$, there exists $\gamma_{k}>0$ such that for any $u \in Z_{k}$ with $\|u\|_{\mathcal{H}_{0}^{s}}=\gamma_{k}$, we have $\mathcal{J}(u) \rightarrow +\infty$.\par 
		By (H1), we get that there exists a constant $C>0$ such that
		$|F(x, u)| \leqslant C\left(1+|u|^{q}\right)$ 
		for any $x \in \overline{\Omega}$ and $u \in \mathbb{R}$. 

		Then, for any $u \in Z_{k} \backslash\{0\}$, we obtain
		\[\begin{aligned}
			\mathcal{J}(u) &=\frac{1}{2}\|u\|^{2}_{\mathcal{H}_{0}^{s}} -\int_{\Omega}F(x,u)dx
			 \geqslant \frac{1}{2}\|u\|_{\mathcal{H}_{0}^{s}}^{2}-C|u|_{q}^{q}-C|\Omega| \\
			& =\frac{1}{2}\|u\|_{\mathcal{H}_{0}^{s}}^{2}-C\left|\frac{u}{\|u\|_{\mathcal{H}_{0}^{s}}}\right|_{q}^{q}\|u\|_{\mathcal{H}_{0}^{s}}^{q}-C|\Omega| \\
			& \geqslant \frac{1}{2}\|u\|_{\mathcal{H}_{0}^{s}}^{2}-C \beta_{k}^{q}\|u\|_{\mathcal{H}_{0}^{s}}^{q}-C|\Omega| 
			 =\|u\|_{\mathcal{H}_{0}^{s}}^{2}\left(\frac{1}{2}-C \beta_{k}^{q}\|u\|_{\mathcal{H}_{0}^{s}}^{q-2}\right)-C|\Omega|
		\end{aligned}\]
		where $\beta_{k}$ is defined as in Lemma \ref{Lamma Fountain} .
		
		 Choosing
		$\gamma_{k}=\left(q C \beta_{k}^{q}\right)^{-1 /(q-2)},$
		it is easy to see that $\gamma_{k} \rightarrow+\infty$ as $k \rightarrow+\infty$, thanks to Lemma \ref{Lamma Fountain} and the fact that $q>2$. As a consequence, we get that for any $u \in Z_{k}$ with $\|u\|_{\mathcal{H}_{0}^{s}}=\gamma_{k}$,
		\[\mathcal{J}(u) \geqslant\|u\|_{\mathcal{H}_{0}^{s}}^{2}\left(1/2-C \beta_{k}^{q}\|u\|_{\mathcal{H}_{0}^{s}}^{q-2}\right)-C|\Omega|=(1/2-1/q) \gamma_{k}^{2}-C|\Omega| \rightarrow+\infty \text{ as } k \rightarrow+\infty.\]
		
	 It remains to verify the assumption (i). For this, we just need to prove that there exists $\rho_{k}>0$ such that for any $u \in Y_{k}$ with $\|u\|_{\mathcal{H}_{0}^{s}}=\rho_{k}$, we have $\mathcal{J}(u) \leqslant0$.\par 
	 By (H3), we have 
	 \[F(x,t)\geqslant \frac{1}{C_{k,2}^{2}} t^{2}-B_{k,2}\]
	 where $C_{k,2}$ is a positive constant given in \eqref{finite dim eq norm} with $q=2$, and $B_{k,2}>0$ is a constant related to $C_{k,2}$. \par 
	 Then, for any $u\in Y_{k}$
	 \[\begin{aligned}
	 	\mathcal{J}(u) &=\frac{1}{2}\|u\|^{2}_{\mathcal{H}_{0}^{s}} -\int_{\Omega}F(x,u)dx
	 	\leqslant \frac{1}{2}\|u\|_{\mathcal{H}_{0}^{s}}^{2}-\frac{1}{C_{k,2}^{2}}|u|_{2}^{2}+B_{k,2} |\Omega| \\
	 	& \leqslant \frac{1}{2}\|u\|_{\mathcal{H}_{0}^{s}}^{2}-\|u\|_{\mathcal{H}_{0}^{s}}^{2}+B_{k,2}|\Omega|
	 	=-\frac{1}{2}\|u\|_{\mathcal{H}_{0}^{s}}^{2}+B_{k,2}|\Omega|.
	 \end{aligned}\]\par 
	 Let $\|u\|_{\mathcal{H}_{0}^{s}}=\rho_{k}>\gamma_{k}>0$ large enough. We get that $\mathcal{J}(u) \leqslant 0$.\par 
	In conclusion, we can deduce that  $\mathcal{J}$ has infinitely many critical points $\{u_{j}\}_{j\in \mathbb{N}}$ and $\mathcal{J}(u_{j}) \rightarrow +\infty$ as $j\rightarrow +\infty$. 		
	\end{proof}

	\subsection{Positive and negative energy solutions of concave-convex problems} In this subsection we consider the problem
	\begin{equation*}
		\left\{
		\begin{array}{ll}
			(-\Delta)^{s} u=\lambda|u|^{p-2}u+ \mu g(x,u) & \text{ in }\Omega,\\
			u=0 &  \text{in } \mathbb{R}^n \setminus \Omega 
		\end{array}
		\right.
	\end{equation*}
	where $1<p<2$, $\lambda,\mu >0$ and $g(x,t)$ is a continuous function satisfying (H1)-(H4) and (S). We obtain infinitely many solutions with both positive energy and negative energy.\par

	The infinitely many positive energy solutions can be found by an analogous argument as Theorem~\ref{infinitely many sol Theorem}. To seek infinitely many negative energy solutions, we need Dual Fountain Theorem (see \cite[Theorem 3.18]{Willembook}). Noticing that $\mathcal{H}_{0}^{s}$ satisfies (F1) and $\mathcal{J}_{\lambda,\mu} \in C^{1}(\mathcal{H}_0^s, \mathbb{R})$ is invariant, we just need to verify the geometric condition (B1)  and compactness condition (B2): for every $k \geqslant k_{0}$,
	\begin{itemize}
		\item [\textbf{(B1)}]  there exists $\rho_{k}>\gamma_{k}>0$ such that
		    \begin{itemize}
		    	\item [(i)] $ b_{k}:=\max \left\{\mathcal{J}_{\lambda,\mu}(u):u \in Y_{k}, \|u\|_{\mathcal{H}_{0}^{s}}=\gamma_{k}\right\} <0$, 
		    	\item [(ii)] $ a_{k}:= \inf\left\{\mathcal{J}_{\lambda,\mu}(u):u \in Z_{k}, \|u\|_{\mathcal{H}_{0}^{s}}=\rho_{k}\right\} \geqslant 0$,
		    	\item [(iii)] $ d_{k}:=\inf \left\{\mathcal{J}_{\lambda,\mu}(u):u \in Z_{k}, \|u\|_{\mathcal{H}_{0}^{s}} \leqslant \rho_{k}\right\} \rightarrow 0, \, k \rightarrow +\infty$;
		    \end{itemize}
	    \item [\textbf{(B2)}] $\mathcal{J}_{\lambda,\mu}$ satisfies the $(P S)_{c}^{*}$ condition for every $c \in\left[d_{k_{0}}, 0)\right.$.
	\end{itemize}
	
	Here the $(PS)_{c}^{*}$ condition (with respect to $(Y_{n})$, $Y_{n} \rightarrow \mathcal{H}_{0}^{s}$ as $n\rightarrow +\infty$)  is: if any sequence $\{u_{n_{j}}\} \subset \mathcal{H}_{0}^{s}$ such that
	\[n_{j} \rightarrow +\infty, u_{n_{j}} \in Y_{n_{j}}, \mathcal{J}_{\lambda,\mu}(u_{n_{j}}) \rightarrow c,\left.\mathcal{J}_{\lambda,\mu} \right|_{Y_{n_{j}}} ^{\prime}(u_{n_{j}}) \rightarrow 0\]
	contains a convergent subsequence.
	
	\begin{proof}[\bf{Proof of Theorem \ref{infinite positive and negative energy sol theorem}}]
		Since $f_{\lambda,\mu}(x,u)=\lambda|u|^{p-2}u + \mu g(x,u)$ with $1<p<2$, and $g$ satisfies (H1), (H3) and (H4), we can deduce that $f_{\lambda,\mu}$ also satisfies (H1), (H3) and (H4).
		Noticing $f_{\lambda,\mu}$ satisfies (S), the existence of $\{u_k\}_{k\in \mathbb{N}}$ is proved as a corollary of Theorem \ref{infinitely many sol Theorem}. To seek $\{v_k\}$, we start by proving the geometric feature (B1) of $\mathcal{J}_{\lambda,\mu}$.\par 
		\textbf{Step 1} We first prove that $\exists~k_{0}\in \mathbb{R}$ such that for any $k\geqslant  k_{0}$, $\exists ~\rho_{k}>0$, 
		\[a_{k}= \inf\left\{\mathcal{J}_{\lambda,\mu}(u):u \in Z_{k}, \|u\|_{\mathcal{H}_{0}^{s}}=\rho_{k}\right\} \geqslant 0.\]

		By \eqref{estimate of F}, Lemmata~\ref{Lemma:X is embedded in fractional Sobolev space} and \ref{Lamma Fountain}, we have, for any $u\in Z_{k}$
		\begin{equation}\label{15}
			\begin{aligned}
				\mathcal{J}_{\lambda,\mu}(u)&= \frac{1}{2} \|u\|_{\mathcal{H}_{0}^{s}}^{2}-\frac{\lambda}{p}|u|^{p}_{p}-\mu\int_{\Omega}G(x,u)dx\\
				&= \frac{1}{2} \|u\|_{\mathcal{H}_{0}^{s}}^{2}-\frac{\lambda}{p}\left|\frac{u}{\|u\|_{\mathcal{H}_{0}^{s}}}\right|^{p}_{p} \|u\|_{\mathcal{H}_{0}^{s}}^{p}-\mu\int_{\Omega}G(x,u)dx\\
				&\geqslant \frac{1}{2} \|u\|_{\mathcal{H}_{0}^{s}}^{2}-\frac{\lambda}{p} \beta_{k}^{p} \|u\|_{\mathcal{H}_{0}^{s}}^{p}-\mu\int_{\Omega} \left(\varepsilon|u|^{2}+\delta(\varepsilon)|u|^{q}\right)dx\\
				&\geqslant \|u\|_{\mathcal{H}_{0}^{s}}^{2}\left(1/2 -\lambda/p \beta_{k}^{p} \|u\|_{\mathcal{H}_{0}^{s}}^{p-2} -\mu\varepsilon C_{2} -\mu C_{q}(\varepsilon)\|u\|_{\mathcal{H}_{0}^{s}}^{q-2} \right).
			\end{aligned}
		\end{equation}
		
		Taking $\varepsilon= \varepsilon_{0}:= \frac{1}{8\mu C_{2}}$, $R= \left( \frac{1}{8\mu C_{q}(\varepsilon_0)}\right)^{1/(q-2)}$, we know that for any $u\in Z_{k}$ with $ \|u\|_{\mathcal{H}_{0}^{s}} \leqslant R$,
		\begin{equation}\label{17}
			\mathcal{J}_{\lambda,\mu}(u) \geqslant  \|u\|_{\mathcal{H}_{0}^{s}}^{2} \left(\frac{1}{4}  -\frac{\lambda}{p} \beta_{k}^{p} \|u\|_{\mathcal{H}_{0}^{s}}^{p-2}\right).
		\end{equation}
		
		Taking $\|u\|_{\mathcal{H}_{0}^{s}} = \rho_{k}:= \min \left\{ R, \left( \frac{4\lambda \beta_{k}^{p}}{p} \right)^{\frac{1}{2-p}} \right\}$, we can deduce that $\mathcal{J}_{\lambda,\mu}(u) \geqslant 0$. Since $1<p<2$, we have $\rho_{k} \rightarrow 0$ as $k\rightarrow +\infty$ by Lemma \ref{Lamma Fountain}.\par 
		As a consequence, for any $k \geqslant k_{0}$, there exists $\rho_{k}>0$ such that (ii) is valid.\par 
		\textbf{Step 2} We next prove $ d_{k}:=\inf \left\{\mathcal{J}_{\lambda,\mu}(u):u \in Z_{k}, \|u\|_{\mathcal{H}_{0}^{s}} \leqslant \rho_{k}\right\} \rightarrow 0, k \rightarrow +\infty.$\par
		Indeed, for any $u \in Z_{k}$ with $\|u\|_{\mathcal{H}_{0}^{s}} \leqslant \rho_{k} \leqslant R$, we have 
		$\mathcal{J}_{\lambda,\mu}(u) \geqslant -\frac{\lambda}{p} \beta_{k}^{p} \|u\|_{\mathcal{H}_{0}^{s}}^{p} \geqslant -\frac{\lambda}{p} \beta_{k}^{p} R^{p}$
		by \eqref{17}. Then we have 
		\[0=\mathcal{J}_{\lambda,\mu}(0) \geqslant \inf\left\{\mathcal{J}_{\lambda,\mu}(u):u \in Z_{k}, \|u\|_{\mathcal{H}_{0}^{s}} \leqslant \rho_{k}\right\} \geqslant -\frac{\lambda}{p} \beta_{k}^{p} R^{p} \rightarrow 0 \]
		as $k \rightarrow +\infty$ by Lemma \ref{Lamma Fountain}. Thus $d_{k}\rightarrow 0$ as $k \rightarrow +\infty$.\par 
		\textbf{Step 3} As for (i), we just need to prove: for any $u\in Y_{k}, \|u\|_{\mathcal{H}_{0}^{s}}=\gamma_{k}$, $\mathcal{J}_{\lambda,\mu} < 0$. Indeed, for any $u\in Y_{k}$, we have
		\[\begin{aligned}
			\mathcal{J}_{\lambda,\mu}(u)&= \frac{1}{2} \|u\|_{\mathcal{H}_{0}^{s}}^{2}-\frac{\lambda}{p}|u|^{p}_{p}-\mu\int_{\Omega}G(x,u)dx
			\leqslant \frac{1}{2} \|u\|_{\mathcal{H}_{0}^{s}}^{2}-\frac{\lambda}{p}|u|^{p}_{p}+\mu \varepsilon|u|_{2}^{2}+\mu \delta(\varepsilon)|u|_{q}^{q}\\
			&\leqslant \frac{1}{2} \|u\|_{\mathcal{H}_{0}^{s}}^{2}-\frac{\lambda}{p} C_{k,p}^{p} \|u\|^{p}_{\mathcal{H}_{0}^{s}} +\mu \varepsilon C_{2}\|u\|_{\mathcal{H}_{0}^{s}}^{2} +\mu C_{q}(\varepsilon)\|u\|_{\mathcal{H}_{0}^{s}}^{q}\\
			&= \|u\|^{p}_{\mathcal{H}_{0}^{s}} \left( \left(\frac{1}{2} + \mu  \varepsilon C_{2} \right) \|u\|_{\mathcal{H}_{0}^{s}}^{2-p}-\frac{\lambda}{p} C_{k,p}^{p}+\mu C_{q}(\varepsilon)\|u\|_{\mathcal{H}_{0}^{s}}^{q-p} \right)
		\end{aligned}\]
		by \eqref{estimate of F}, \eqref{finite dim eq norm} and Lemma~\ref{Lemma:X is embedded in fractional Sobolev space}.\par 
		Since $1<p<2<q<2_{s}^{*}$, we have 
		\[\lim_{\|u\|_{\mathcal{H}_{0}^{s}} \rightarrow 0} \left( \left(\frac{1}{2} + \mu  \varepsilon C_{2} \right) \|u\|_{\mathcal{H}_{0}^{s}}^{2-p}-\frac{\lambda}{p} C_{k,p}^{p}+\mu C_{q}(\varepsilon)\|u\|_{\mathcal{H}_{0}^{s}}^{q-p} \right) = -\frac{\lambda}{p} C_{k,p}^{p}.\]\par 
		Thus (i) follows by taking $\gamma_{k}$ sufficiently small.\par 
		Similar to the proofs of Proposition~\ref{PS bounded sequence} and \ref{PSc convergent sequence}, we can deduce that $(P S)_{c}^{*}$ condition is verified.  
		Using Dual Fountain Theorem, we obtain a sequence of solutions $\{v_k\}$ such that  $\mathcal{J}_{\lambda,\mu}(v_k)\rightarrow 0^-$, $k\rightarrow +\infty$.
	\end{proof}

   \begin{proof}[\bf{Proof of Theorem \ref{infinite positive and negative energy sol theorem-2}}]
	When $\mu>0,\lambda \in \mathbb{R}$, we can easily deduce that $f_5$ satisfies (H1), (AR) and (S). Thus, part $(a)$ is shown as a consequence of Theorem \ref{negative energy sol theorem}.\par
	We now prove part $(b)$. Noticing that (B1) can be verified in the same way as in Theorem \ref{infinite positive and negative energy sol theorem}, we only need to check the $(PS)_{c}^{*}$ condition.\par

	Consider the $(PS)_{c}^{*}$ sequence $\{u_{n_{j}}\} \subset \mathcal{H}_{0}^{s}$.
	For $n_{j}$ large enough, we have 
	\[ \mathcal{J}_{\lambda,\mu} (u_{n_{j}}) \leqslant c+1 , \quad \left.\mathcal{J}_{\lambda,\mu} \right|_{Y_{n_{j}}} ^{\prime}(u_{n_{j}}) \leqslant 1.\]
	
	Thus
	\begin{equation}\label{22}
		\begin{aligned}
			&\quad \mathcal{J}_{\lambda,\mu}(u_{n_{j}}) - \frac{1}{q} \left \langle \mathcal{J}_{\lambda,\mu}^{\prime}(u_{n_{j}}), u_{n_{j}} \right \rangle_{\mathcal{H}_{0}^{s}}
			\leqslant c+1+\left| \left \langle \mathcal{J}_{\lambda,\mu}^{\prime}(u_{n_{j}}), u_{n_{j}} \right \rangle_{\mathcal{H}_{0}^{s}} \right|\\
			&\leqslant c+1+ \left\| \mathcal{J}_{\lambda,\mu}^{\prime}(u_{n_{j}}) \right\|_{\mathcal{H}_{0}^{s}} \| u_{n_{j}} \|_{\mathcal{H}_{0}^{s}}
			\leqslant c+1+\| u_{n_{j}} \|_{\mathcal{H}_{0}^{s}}.
		\end{aligned}		
	\end{equation}
	On the other hand,
	\begin{equation}\label{23}
		\begin{aligned}
			&\quad \mathcal{J}_{\lambda,\mu}(u_{n_{j}}) - \frac{1}{q} \left \langle \mathcal{J}_{\lambda,\mu}^{\prime}(u_{n_{j}}), u_{n_{j}} \right \rangle_{\mathcal{H}_{0}^{s}}\\
			&=\left(1/2 -1/q \right)\|u_{n_{j}}\|_{\mathcal{H}_{0}^{s}}^{2}
			-\left(\lambda/p - \lambda/q \right)|u_{n_{j}}|^{p}_{p}
			+\left(\mu/q-\mu/q\right)|u_{n_{j}}|^{q}_{q}\\
			&=\left(\frac{1}{2} - \frac{1}{q} \right)\|u_{n_{j}}\|_{\mathcal{H}_{0}^{s}}^{2}
			-\left(\frac{\lambda}{p} -\frac{\lambda}{q} \right)\left|\frac{u_{n_{j}}}{\|u_{n_{j}}\|_{\mathcal{H}_{0}^{s}}}\right|^{p}_{p}\|u_{n_{j}}\|_{\mathcal{H}_{0}^{s}}^{p}\\
			&\geqslant \left(1/2 -1/q  \right)\|u_{n_{j}}\|_{\mathcal{H}_{0}^{s}}^{2}
			-\left(\lambda/p - \lambda/q \right)\beta_{0}^{p}\|u_{n_{j}}\|_{\mathcal{H}_{0}^{s}}^{p}.
		\end{aligned}		
	\end{equation}\par 
	Combining \eqref{22} and \eqref{23}, we have 
	\[ (1/2 - 1/q )\|u_{n_{j}}\|_{\mathcal{H}_{0}^{s}}^{2} \leqslant c+1+ \| u_{n_{j}} \|_{\mathcal{H}_{0}^{s}} + (\lambda/p -\lambda/q )\beta_{0}^{p}\|u_{n_{j}}\|_{\mathcal{H}_{0}^{s}}^{p}. \]\par 
	We can deduce that $\{u_{n_{j}}\}$ is bounded in $\mathcal{H}_{0}^{s}$. Indeed, we suppose, by contradiction, that up to a subsequence, still denoted by $\{u_{n_{j}}\}$,
	\[ 0<\frac{1}{2} - \frac{1}{q} \leqslant \frac{c+1}{\|u_{n_{j}}\|_{\mathcal{H}_{0}^{s}}^{2}}+ \frac{1}{\|u_{n_{j}}\|_{\mathcal{H}_{0}^{s}}}+ \left(\frac{\lambda}{p} -\frac{\lambda}{q} \right)\beta_{0}^{p} \frac{1}{\|u_{n_{j}}\|_{\mathcal{H}_{0}^{s}}^{2-p}} \rightarrow 0, \text{ as }j\rightarrow +\infty,\]
	 thanks to $1<p<2<q<2^{*}_{s}$. Thus $\{u_{n_{j}}\}$ is bounded in $\mathcal{H}_{0}^{s}$. The convergent subsequence can be obtained as in the proof of Proposition~\ref{PSc convergent sequence}.\par  
	As a consequence, part $(b)$ is proved by Dual Fountain Theorem.
\end{proof}

\section*{Appendix}\setcounter{equation}{0}
\renewcommand{\theequation}{A.\arabic{equation}}
	In this section, we study 
	the weak eigenvalue problem associated to $(-\Delta)^{s}:$
	\begin{equation}\label{eigen problem}
		\left\{\begin{array}{l}
			\langle u, \varphi \rangle_{\mathcal{H}_{0}^{s}}=\lambda \int_{\Omega} u(x) \varphi(x) d x, \quad \forall \varphi \in \mathcal{H}_{0}^{s} \\
			u \in \mathcal{H}_{0}^{s}.
		\end{array}\right.
	\end{equation}
    
    \textbf{Lemma A.1}
		(i) problem \eqref{eigen problem} admits an eigenvalue 
		\[\lambda_{1}=\min \{~\|u\|_{\mathcal{H}_{0}^{s}}^{2}: u \in \mathcal{H}_{0}^{s}, |u|_{2}=1 \}>0.\]
		And there exists a non-trivial function $e_{1} \in \mathcal{H}_{0}^{s}$ such that $|e_{1}|_{2}=1$, which is an eigenfunction corresponding to $\lambda_{1}$, attaining the minimum.

		(ii) the set of the eigenvalues of problem \eqref{eigen problem} is a sequence $\left\{\lambda_{k}\right\}_{k \in \mathbb{N}}$ with 
		\[0<\lambda_{1}\leqslant\lambda_{2} \leqslant \cdots \leqslant \lambda_{k} \leqslant \lambda_{k+1} \leqslant \cdots\]
		and
		$\lambda_{k} \rightarrow+\infty \text { as } k \rightarrow+\infty$.
		Moreover, for any $k \in \mathbb{N}$ the eigenvalues can be characterized as
		\begin{equation}\label{eigenvalue}
			\lambda_{k+1}=\min \{~ \|u\|_{\mathcal{H}_{0}^{s}}^{2}: u \in \mathbb{P}_{k+1}, |u|_{2}=1\}
		\end{equation}
		where
		$\mathbb{P}_{k+1}:=\left\{u \in \mathcal{H}_{0}^{s} \text { s.t. }\langle u, e_{j}\rangle_{\mathcal{H}_{0}^{s}}=0 \quad \forall j=1, \cdots, k\right\}$.
		
		(iii) the sequence $\left\{e_{k}\right\}_{k \in \mathbb{N}}$ of eigenfunctions corresponding to $\lambda_{k}$ is an orthonormal basis of $L^{2}(\Omega)$ and an orthogonal basis of $\mathcal{H}_{0}^{s}$.  
	
	To prove the above lemma, we give the following claim inspired by \cite{MR3002745}.\\
	\textbf{Claim A.1}
		Let $\mathcal{F}: \mathcal{H}_{0}^{s} \rightarrow \mathbb{R}$ be the functional defined as 
		$\mathcal{F}(u)=\frac{1}{2}\|u\|_{\mathcal{H}_{0}^{s}}^{2} .$
		Then 
		
		$(a)$ if $X_* \neq \emptyset$ is a weakly closed subspace of $\mathcal{H}_{0}^{s}$ and $\mathcal{M}_*:=\{u \in$ $\left.X_*:|u|_{2}=1\right\}$, then there exists $u_* \in \mathcal{M}_*$ such that $\min _{u \in \mathcal{M}_*} \mathcal{F}(u)=\mathcal{F}\left(u_*\right)$
		and
		\[\left\langle u_*, \varphi\right\rangle_{\mathcal{H}_{0}^{s}}=\lambda_*(u_*,\varphi), \quad \forall \varphi \in X_*\]
		where $\lambda_*:=2\mathcal{F}(u_*)$. Here and in the sequel, we denote the product on $L^{2}$ by $(\cdot, \cdot )$.\par
		$(b)$ if $\lambda \neq \tilde \lambda$ are two different eigenvalues of problem \eqref{eigen problem}, with eigenfunctions $e$ and $\tilde{e} \in \mathcal{H}_{0}^{s}$, respectively, then
		\begin{equation}\label{12}
			\langle e, \tilde{e}\rangle_{\mathcal{H}_{0}^{s}}=0=(e,\tilde{e}) .
		\end{equation}
		And if $k, h \in \mathbb{N}, k \neq h$, then
		\begin{equation}\label{10}
			\left\langle e_{k}, e_{h}\right\rangle_{\mathcal{H}_{0}^{s}}=0=( e_{k}, e_{h}) .
		\end{equation}\par
		$(c)$ if $e$ is an eigenfunction of problem \eqref{eigen problem} corresponding to an eigenvalue $\lambda$, then
		$ \| e \|_{\mathcal{H}_{0}^{s}}=\lambda|e|_{2}^{2}$.

	\begin{proof}
		We only prove part $(a)$. Let $\{u_{j}\}_{j\in \mathbb{N}} \subset \mathcal{M}_*$ be a minimizing sequence for $\mathcal{F}$, that is
		\[\lim \limits_{j \rightarrow \infty} \mathcal{F} (u_{j})=\inf _{\mathcal{M}_*}\mathcal{F}(u)=\inf _{\mathcal{M}_*}\frac{1}{2}\|u\|_{\mathcal{H}_{0}^{s}}^{2} \geqslant 0.\]\par
		Then $\{u_{j}\}$ is bounded in $\mathcal{H}_{0}^{s}$. By Lemma \ref{X10 compact embedded in Lq}, up to a subsequence, still denoted by $\{u_{j}\}$, there exists $u_{*} \in \mathcal{H}_{0}^{s}$ such that
		\[  u_{j} \rightharpoonup u_{*} \text { in } \mathcal{H}_{0}^{s} \text{ and }
			u_{j} \rightarrow u_{*} \text { in } L^{2}(\Omega) .\]\par 
		Since $|u_{j}|_{2}=1$,we know that $|u_{*}|_{2}=1$ and $u_{*} \in \mathcal{M}_*$. According to the weak lower semi-continuity of the norm, we deduce that
		\[\lim \limits_{j \rightarrow \infty} \mathcal{F}(u_{j})  
		\geqslant \mathcal{F}\left(u_*\right) \geqslant \inf _{u \in \mathcal{M}_*} \mathcal{F}(u),\]
		which implies that $\mathcal{F}\left(u_{*}\right)=\inf _{\mathcal{M_*}} \mathcal{F}(u)$. 
		
		Let $\varepsilon \in(-1,1)$ and $v \in X_*$. Define $u_{\varepsilon}=\frac{u_{*}+\varepsilon v}{|u_{*}+\varepsilon v|_{2}}$. Then $u_{\varepsilon} \in \mathcal{M}_*$ and
		\[
		\begin{aligned}
			2\mathcal{F}(u_{\varepsilon}) &=\left \langle u_{\varepsilon},u_{\varepsilon} \right\rangle_{\mathcal{H}_{0}^{s}} =\frac{\|u_{*}\|_{\mathcal{H}_{0}^{s}}^{2}+2 \varepsilon\langle u_{*}, v\rangle_{\mathcal{H}_{0}^{s}}+\varepsilon^{2}\|v\|_{\mathcal{H}_{0}^{s}}^{2}}{1+2 \varepsilon \int_{\Omega} u_{*} v d x+\varepsilon^{2}|v|_{2}^{2}} \\
			& =\frac{\|u_{*}\|_{\mathcal{H}_{0}^{s}}^{2}+2 \varepsilon \langle u_{*}, v\rangle_{\mathcal{H}_{0}^{s}}+\varepsilon^{2}\|v\|_{\mathcal{H}_{0}^{s}}^{2}}{1-4 \varepsilon^{2}(\int_{\Omega} u_{*} v d x)^{2}+2 \varepsilon^{2}|v|_{2}^{2}+\varepsilon^{4}|v|_{2}^{4}}\left(1-2 \varepsilon \int_{\Omega} u_{*} v d x+\varepsilon^{2}|v|_{2}^{2}\right) \\
			& \leqslant \frac{\|u_{*}\|_{\mathcal{H}_{0}^{s}}^{2}+2 \varepsilon\langle u_{*}, v\rangle_{\mathcal{H}_{0}^{s}}+\varepsilon^{2}\|v\|_{\mathcal{H}_{0}^{s}}^{2}}{(1-\varepsilon^{2}|v|_{2}^{2})^{2}}\left(1-2 \varepsilon \int_{\Omega} u_{*} v d x+\varepsilon^{2}|v|_{2}^{2}\right) \\
			& =\frac{1}{(1-\varepsilon^{2}|v|_{2}^{2})^{2}}\left(2 \mathcal{F}\left(u_{*}\right)+2 \varepsilon(\left\langle u_{*}, v\right\rangle_{\mathcal{H}_{0}^{s}}-2 \mathcal{F}\left(u_{*}\right) \int_{\Omega} u_{*} v d x)+o(\varepsilon)\right) ,
		\end{aligned}\]
		where the above inequality  holds because $\int_{\Omega} u_* v dx \leqslant \left| u_*\right|_2|v|_2=|v|_{2}.$\par
		Let $\varepsilon$ be sufficiently small, we have $\left\langle u_{*}, v\right\rangle_{\mathcal{H}_{0}^{s}}-2 \mathcal{F}\left(u_{*}\right) (u_{*}, v)=0.$		
	\end{proof}
	
	\begin{proof}[\bf {Proof of Lemma~A.1}]
		We argue as in the proof of \cite[Proposition 9]{MR3002745}. For the sake of completeness of the paper, we provide a brief sketch.\par  
		(i) Applying Claim~A.1-(a) with $X_*:=\mathcal{H}_{0}^{s}$, we obtain Lemma~A.1-(i).\par
		(ii) We divide the proof of Lemma~A.1-(ii) into the following five steps.\par 
		\textbf{Step 1}. Since $\mathbb{P}_{k+1} \subseteq \mathbb{P}_{k} \subseteq \mathcal{H}_{0}^{s}$, we have that
		\[0<\lambda_{1} \leqslant \lambda_{2} \leqslant \cdots \leqslant \lambda_{k} \leqslant \lambda_{k+1} \leqslant \cdots.\]\par 
		\textbf{Step 2}. Applying Claim~A.1-(a) with $X_*:=\mathbb{P}_{k+1}$,  which is weakly closed, the minimum of $\lambda$ exists and it is attained at some $e_{k+1} \in \mathbb{P}_{k+1}$. Also, we have
		\begin{equation}\label{8}
			\langle e_{k+1}, \varphi \rangle_{\mathcal{H}_{0}^{s}}=\lambda_{k+1} ( e_{k+1}, \varphi) \quad \forall \varphi \in \mathbb{P}_{k+1}.
		\end{equation}\par 
		\textbf{Step 3}. In order to show that $\lambda_{k+1}$ is an eigenvalue with eigenfunction $e_{k+1}$, we need to show that
		formula \eqref{8} holds for any $\varphi \in \mathcal{H}_{0}^{s}$, not only in $\mathbb{P}_{k+1}$.\par 
		Since  $\lambda_{1}$ is an eigenvalue, as shown in (i), we argue recursively, assuming that the claim holds for $1, \cdots, k$ and proving it for $k+1$. We use the direct sum decomposition
		$\mathcal{H}_{0}^{s}=\operatorname{span}\left\{e_{1}, \cdots, e_{k}\right\} \oplus \mathbb{P}_{k+1}.$\par 
		Thus, given any $\varphi \in \mathcal{H}_{0}^{s}$, we write $\varphi=\varphi_{1}+\varphi_{2}$, with $\varphi_{2} \in \mathbb{P}_{k+1}$
		and $\varphi_{1}=\sum_{i=1}^{k} c_{i} e_{i}$,
		for some $c_{1}, \cdots, c_{k} \in \mathbb{R}$. Then, from \eqref{8} tested with $\varphi_{2}=\varphi-\varphi_{1}$, we know that
		\begin{equation}\label{9}
			\begin{aligned}
				&\quad \langle e_{k+1}, \varphi \rangle_{\mathcal{H}_{0}^{s}}-\lambda_{k+1} ( e_{k+1}, \varphi) =\langle e_{k+1}, \varphi_1\rangle_{\mathcal{H}_{0}^{s}}-\lambda_{k+1} ( e_{k+1}, \varphi_{1})\\&=\sum \nolimits_{i=1}^{k} c_{i}\left[\langle e_{k+1},e_i\rangle_{\mathcal{H}_{0}^{s}}-\lambda_{k+1} (e_{k+1}, e_{i})\right].
			\end{aligned}
		\end{equation}\par
		Furthermore, by inductive assumption, $\lambda_i$ is an eigenvalue and $e_i$ is the corresponding eigenfunction. Testing the eigenvalue equation for $e_{i}$ by test function $e_{k+1}$ for $i=$ $1, \cdots, k$ and recalling that $e_{k+1}$ $\in \mathbb{P}_{k+1}$, we see that
		$	0=\langle e_{k+1},e_i\rangle_{\mathcal{H}_{0}^{s}}=\lambda_{i} (e_{k+1}, e_{i}).$
		
		Thanks to \eqref{10} and $\lambda_i>0$, we have
		$\langle e_{k+1},e_i\rangle_{\mathcal{H}_{0}^{s}}=0=(e_{k+1}, e_{i})$
		for any $i=1, \cdots, k$. By plugging this into \eqref{9}, we conclude that \eqref{8} holds true for any $\varphi \in \mathcal{H}_{0}^{s}$, that is $\lambda_{k+1}$ is an eigenvalue with eigenfunction $e_{k+1}$.\par
		 
		\textbf{Step 4}. Now we prove $\lambda_{k} \rightarrow +\infty$. 
		
		Suppose, by contradiction, that $\lambda_{k} \rightarrow c$ for some constant $c \in \mathbb{R}$. Then $\lambda_{k}$ is bounded in $\mathbb{R}$. Since $\left\|e_{k}\right\|_{\mathcal{H}_{0}^{s}}^{2}=\lambda_{k}$ by Claim~A.1-(c) , we deduce by Lemma \ref{X10 compact embedded in Lq} that there is a subsequence for which
		\[e_{k_{j}} \rightarrow e_{\infty} \quad \text { in } L^{2}(\Omega)\]
		as $k_{j} \rightarrow+\infty$, for some $e_{\infty} \in L^{2}(\Omega)$. In particular, $e_{k_{j}}$ is a Cauchy sequence in $ L^{2}(\Omega) $. 
		But, from \eqref{10}, $e_{k_{j}}$ and $e_{k_{i}}$ are orthogonal in $L^{2}(\Omega)$, so
		\[|e_{k_{j}}-e_{k_{i}}|_{2}^{2}=|e_{k_{j}}|_{2}^{2}+|e_{k_{i}}|_{2}^{2}=2 .\]\par
		This is a contradiction, which implies $\lambda_{k} \rightarrow +\infty$.\par
		\textbf{Step 5}. The sequence of eigenvalues constructed in \eqref{eigenvalue} exhausts all the eigenvalues of the problem, i.e. that any eigenvalue of problem \eqref{eigen problem} can be written in the form \eqref{eigenvalue}. 
		
		By contradiction, we suppose that there exists an eigenvalue $\lambda \notin\left\{\lambda_{k}\right\}_{k \in \mathbb{N}}$,
		and let $e \in \mathcal{H}_{0}^{s}$ be an eigenfunction relative to $\lambda$, normalized so that $|e|_{2}=1$. \par 
		By Claim~A.1-(c), we have that $2 \mathcal{F}(e)=\|e\|_{\mathcal{H}_{0}^{s}}^{2} =\lambda$. 
		Since $\lambda \notin \left\{\lambda_{k}\right\}_{k\in \mathbb{N}}$ and $\lambda_{k} \rightarrow \infty$, then there exists $ k\in \mathbb{N}$ such that $\lambda_{k}<\lambda<\lambda_{k+1}.$\par 
		We claim that $e\notin \mathbb{P}_{k+1}$.
		Indeed, if $e\in \mathbb{P}_{k+1}$, then $\lambda=2\mathcal{F}(e)\geqslant \lambda_{k+1}$. This is a contradiction. 
		
		Since $e\notin \mathbb{P}_{k+1}$, then there exists $ i\in {1, \cdots, k}$ such that $\langle e,e_i\rangle_{\mathcal{H}_{0}^{s}} \neq 0$. This is a contradiction with Claim~A.1-(b). So that
		$\lambda \in \left\{\lambda_{k}\right\}_{k\in \mathbb{N}}$.
		
		(iii) We give the proof of Lemma~A.1-(iii) with the following three steps.\par
		\textbf{Step1} The orthogonality follows from Claim~A.1-(b).\par
		\textbf{Step2} $\left\{e_{k}\right\}_{k \in \mathbb{N}}$ is a basis of $\mathcal{H}_{0}^{s}$. 
		
		We first claim: 
			if $v \in \mathcal{H}_{0}^{s}$ is such that $\left\langle v, e_{k}\right\rangle_{\mathcal{H}_{0}^{s}}=0$ for any $k \in \mathbb{N}$
			, then $v \equiv 0$.
			 
		We define $E_{i}:=e_{i} /\left\|e_{i}\right\|_{\mathcal{H}_{0}^{s}}$ and, for a given $f \in \mathcal{H}_{0}^{s}$, define $v_{j}:=f-f_{j}$ where
		\[f_{j}:=\sum \nolimits_{i=1}^{j}\left\langle f, E_{i}\right\rangle_{\mathcal{H}_{0}^{s}} E_{i} .\]\par
		By the orthogonality of $\left\{e_{k}\right\}_{k \in \mathbb{N}}$ in $\mathcal{H}_{0}^{s}$,
		\[\begin{aligned}
			0& \leqslant  \|v_{j}\|_{\mathcal{H}_{0}^{s}}^{2}=\langle v_{j}, v_{j}\rangle_{\mathcal{H}_{0}^{s}} 
			 =\|f\|_{\mathcal{H}_{0}^{s}}^{2}+\|f_{j}\|_{\mathcal{H}_{0}^{s}}^{2}-2\langle f, f_{j}\rangle_{\mathcal{H}_{0}^{s}}\\
			&=\|f\|_{\mathcal{H}_{0}^{s}}^{2}+\langle f_{j}, f_{j}\rangle_{\mathcal{H}_{0}^{s}}-2 \sum\nolimits_{i=1}^{j}\langle f, E_{i}\rangle_{\mathcal{H}_{0}^{s}}^{2} 
			=\|f\|_{\mathcal{H}_{0}^{s}}^{2}-\sum \nolimits_{i=1}^{j}\langle f, E_{i}\rangle_{\mathcal{H}_{0}^{s}}^{2} .
		\end{aligned}\]\par
		Therefore,  $\sum \nolimits_{i=1}^{j}\left\langle f, E_{i}\right\rangle_{\mathcal{H}_{0}^{s}}^{2} \leqslant\|f\|_{\mathcal{H}_{0}^{s}}^{2}$ for any $j \in \mathbb{N}$.  
		And so
		$\sum\nolimits_{i=1}^{+\infty}\left\langle f, E_{i}\right\rangle_{\mathcal{H}_{0}^{s}}^{2}$ is a convergent series, i.e.,
		\[\tau_{j}:=\sum \nolimits_{i=1}^{j}\left\langle f, E_{i}\right\rangle_{\mathcal{H}_{0}^{s}}^{2}\] is a Cauchy sequence in $\mathbb{R}$.\par
		Moreover, using again the orthogonality of $\left\{e_{k}\right\}_{k \in \mathbb{N}}$ in $\mathcal{H}_{0}^{s}$, we see that, if $J>j$,
		\[\left\|v_{J}-v_{j}\right\|_{\mathcal{H}_{0}^{s}}^{2}=\left \|\sum \nolimits_{i=j+1}^{J}\left\langle f, E_{i}\right\rangle_{\mathcal{H}_{0}^{s}} E_{i}\right\|_{\mathcal{H}_{0}^{s}}^{2} 
		=\sum \nolimits_{i=j+1}^{J}\left\langle f, E_{i}\right\rangle_{\mathcal{H}_{0}^{s}}^{2}=\tau_{J}-\tau_{j} .\]
		Thus $v_{j}$ is a Cauchy sequence in $\mathcal{H}_{0}^{s}$, and there exists $v \in \mathcal{H}_{0}^{s}$ such that
		\[v_{j} \rightarrow v \text { in } \mathcal{H}_{0}^{s} \text { as } j \rightarrow+\infty \text {. }\]\par
		For $j \geqslant k$,
		$\langle v_{j}, E_{k}\rangle_{\mathcal{H}_{0}^{s}}=\left\langle f, E_{k}\right\rangle_{\mathcal{H}_{0}^{s}}-\langle f_{j}, E_{k}\rangle_{ \mathcal{H}_{0}^{s}}
		=0 .$
		
		Taking $j \rightarrow \infty$, we have $\left\langle v, E_{k}\right\rangle_{\mathcal{H}_{0}^{s}}=0$ for any $k \in \mathbb{N}$. So $v=0$. Thus,
		\[f=\sum \nolimits_{i=1}^{+\infty}\left\langle f, E_{i}\right\rangle_{\mathcal{H}_{0}^{s}} E_{i} ,\]
		that is to say $\left\{e_{k}\right\}_{k \in \mathbb{N}}$ is a basis in $\mathcal{H}_{0}^{s}$.
		
		\textbf{Step 3}  $\{e_{k}\}_{k \in \mathbb{N}}$ is a basis of $L^{2}(\Omega)$.
		
		 Take $v \in L^{2}(\Omega)$ and let $v_{j} \in C_{c}^{\infty}(\Omega)$ be such that $|v_{j}-v|_{2} \leqslant 1 / j$. Since $\left\{e_{k}\right\}_{k \in \mathbb{N}}$ is a basis for $\mathcal{H}_{0}^{s}$, there exists $k_{j} \in \mathbb{N}$ and a function $w_{j}$, belonging to $\operatorname{span}\{e_{1}, \cdots, e_{k_{j}}\}$ such that $\|v_{j}-w_{j}\| _{\mathcal{H}_{0}^{s}} \leqslant \frac{1}{j} .$ 
		By Lemma \ref{equivalent norm}, we have
		\[ |v-w_{j}|_{2} \leqslant|v-v_{j}|_{2}+|v_{j}-w_{j}|_{2} \leqslant \frac{1}{j}+C\left\|v_{j}-w_{j}\right\|_{\mathcal{H}_{0}^{s}}\leqslant(C+1) / j .\]
		This shows that the sequence $\left\{e_{k}\right\}$ of eigenfunctions of \eqref{eigen problem} is a basis of $L^{2}(\Omega)$. 
	\end{proof}

\section*{Acknowledgment}	
	 Xifeng Su is supported by National Natural Science Foundation of China (Grant No. 12371186, 11971060).

\bibliographystyle{is-abbrv}	
	\bibliography{reference}
	
\end{document}